\documentclass[11pt]{article}

\usepackage{subcaption}
\usepackage{booktabs}
\usepackage{hyperref}
\usepackage{amsmath}
\usepackage{amsfonts}
\usepackage{amssymb}
\usepackage{mathtools}
\usepackage{amsthm}
\usepackage{a4,a4wide}
\usepackage[dvipdfmx]{graphicx}
\usepackage{color}
\numberwithin{equation}{section}
\numberwithin{figure}{section}
\usepackage{hyperref}

\DeclarePairedDelimiter\abs{\lvert}{\rvert}%

\newcommand{\Real}{\mathbb{R}}

\newcommand{\cA}{\mathcal{A}}
\newcommand{\cB}{\mathcal{B}}
\newcommand{\cH}{\mathcal{H}}

\newcommand{\p}{\partial}

\newcommand{\Om}{\varOmega}
\newcommand{\La}{\lambda}

\newcommand{\X}{\mathcal{X}}

\newcommand{\U}{\mathcal{U}} 
\newcommand{\V}{\mathcal{V}}

\newcommand{\M}{\mathcal{M}}
\newcommand{\N}{\mathcal{N}}
\newcommand{\W}{\mathcal{W}}

\newcommand{\D}{\mathcal{D}}

\newcommand{\Na}{\nabla}

\newcommand{\pt}{\p_t}

\usepackage{amssymb}
\usepackage{epsfig}
\newtheorem{theorem}{Theorem} 

\newtheorem{lemma}[theorem]{Lemma}
\newtheorem{proposition}[theorem]{Proposition}
\newtheorem{definition}[theorem]{Definition}

\usepackage{mathrsfs}
\DeclareMathAlphabet{\mathscrbf}{OMS}{mdugm}{b}{n}

\providecommand{\keywords}[1]{\textbf{\textit{Key words.}} #1}
\providecommand{\MSC}[1]{\textbf{\textit{MSC}} #1}

\title{Multiscale Galerkin approximation scheme for a system of quasilinear parabolic equations}
\author{
Ekeoma R. Ijioma\thanks{MACSI, Department of Mathematics and Statistics, University of Limerick, Ireland.~(Email: e.r.ijioma@gmail.com)} 
~and~Stephen E. Moore\thanks{Katholische Hochschulgemeinde der Di\"ozese Linz, Petrinumstrasse 12/8/D220, A-4040, Linz.~(Email: moorekwesi@gmail.com)} }

\begin{document}

\maketitle

\begin{abstract}
We discuss a multiscale Galerkin approximation scheme for a system of coupled quasilinear parabolic equations. These equations arise from the upscaling of a pore scale filtration combustion model under the assumptions of large Damkh\"oler number and small P\'eclet number. The upscaled model consists of a heat diffusion equation and a mass diffusion equation in the bulk of a macroscopic domain. The associated diffusion tensors are bivariate functions of temperature and concentration and provide the necessary coupling conditions to elliptic-type cell problems. These cell problems are characterized by a reaction-diffusion phenomenon with nonlinear reactions of Arrhenius type at a gas-solid interface. We discuss the wellposedness of the quasilinear system and establish uniform estimates for the finite dimensional approximations. Based on these estimates, the convergence of the approximating sequence is proved. The results of numerical simulations demonstrate, in suitable temperature regimes, the potential of solutions of the upscaled model to mimic those from porous media combustion. Moreover, distinctions are made between the effects of the microscopic reaction-diffusion processes on the macroscopic system of equations and a purely diffusion system.
\end{abstract}

\smallskip
\MSC{74Q05, 34A45, 80A25, 37M05}

\section{Introduction}
\label{sec:introduction}
In this paper, we consider a quasilinear parabolic system that arises in the modeling of a fast exothermic chemical process involving a reactive porous medium. This scenario is prevalent in the area of filtration combustion in porous media, which can be seen in such important processes as in the combustion of fixed bed reactors \cite{Andrez99}, waste incinerators \cite{MAHMOUDI16, RYU07} and in smoldering combustion process with potential to transit to flaming; see, e.g., \cite{DODD12, Ijioma2015b,Ijioma2015c} and all references therein. However, there are other areas of physical modeling in porous media in which the governing equations are of quasilinear parabolic type. These include, but are not limited to the modeling of biofilm growths \cite{Blessing15}, solute dispersion in porous media \cite{Harsha}, etc.

The physical situation of interest is the following: a gaseous mixture containing an oxidizer infiltrates a porous medium predominantly by means of diffusion, assuming convective transport is negligible at the level of description. Then, the oxidizer gas reacts (in a non-premixed manner) with the fuel (solid) surface in the presence of heat to produce char (solid product) and the heat that sustains the process. We assume that the concentration of the oxidizer and the concentration of the solid product slowly changes at a scale, which is much larger than the scale at which the chemical reaction takes place. At the chemical reaction scale of the interaction, the reaction is assumed to be fast. This is governed by a nonlinear and first order Arrhenius-type kinetics, which couples a mass and a heat transport problems at the surface of the material. The material of interest is associated with a microstructure, which is regarded as a representative unit cell of the material. The unit cell consists of two distinct parts: a \emph{solid part} and a \emph{gas-filled part}. Actually, this cell serves as a starting point in the construction of a mesoscopic description with periodically varying functions and parameters posed in a domain that is composed of a scaled and periodically translated copies of the representative cell. We point out that this description is inherent in the mathematical theory of periodic homogenization; see, e.g. \cite{Bakhvalov1989, Bensoussan78,Cioranescu99}.

Furthermore, to be able to handle the complex chemical process taking place in the domain with rapidly oscillating properties, an averaged description of the process is often required; for example, the homogenization method based on multiple scale expansions \cite{Ijioma13}, when applied on the mesoscopic problem, results in an effective quasilinear parabolic problem. The latter problem is peculiar since it retains some of the attributes of the microstructure on the derived macroscopic description. That is, the microstructure can be seen as a point in the macroscopic domain where mass and heat exchange occur through a diffusion mechanism. In other words, variations in heat and mass diffusions across the macroscopic domain is influenced by reaction-diffusion processes in the unit cell. The interplay between these processes occurring at the distinct levels of description are linked through nonlinear diffusion coefficients with respect to the macroscopic variables; specifically, at each macroscopic point, a coupled microscopic reaction-diffusion elliptic problems are solved in the unit cells--one for the heat problem and another for the diffusion problem. In spite of the advantages of the reducing the complexity of the original problem by means of the homogenization method, the complexity of the present problem is given by an increased size of the system, since for each macroscopic point a coupled system of cell equations has to be solved. 

Moreover, the structure of the cell problem arises from nonlinear coupling of chemical reaction at the microscopic level, in which the macroscopic variables enter as parameters. However, besides the micro-macro coupling, the coupled macroscopic problems can be formulated as a single equation incorporating the total enthalpy in terms of the macroscopic variables. The peculiar feature of the studied problem from other two-scale homogenized systems in the literature \cite{Lind18a, Lind18b, MunteanRadu:2010, NeussRadu2007} are two-fold: the first is the coupling of the micro-macro problem via the nonlinear effective diffusion tensors, which vary at each point $x$ of the macroscopic domain and at each time $t$. Next is the strong coupling between the distinct physics at the microscopic level.

Thus, our objective is to provide a Galerkin based approximation scheme that uses the structure of the quasilinear system of equations coupled to elliptic boundary value problems in the formulation of finite-dimensional approximations. The function spaces for the approximation of the cell problems consist of tensor products of functions on the macroscopic domain and on a reference unit cell. The use of such tensor products is inspired by the analysis discussed in \cite{MunteanRadu:2010}. A much recent development on the use of Galerkin approximations for multiscale problems can be found in \cite{Lind18a, Lind18b}. In our context, the macroscopic variable enters the cell problems as a parameter; hence, to pass to the limit in the finite-dimensional approximates, no additional compactness arguments for the derivatives of the cell functions in terms of the macroscopic variable $x$ are necessary. 

This paper is organized as follows. In Section \ref{sec:setting}, the mathematical setting of the problem is described and the necessary assumption on functions and data are stated. Section \ref{galerkinapprox} introduces the Galerkin approximation function spaces, the spatial discretization and the prove of uniform estimates, which assure the compactness of the finite-dimensional approximations. In Section \ref{convergence}, the convergence of the Galerkin approximations is given and this is followed by some numerical experiments to demonstrate the multiscale character of the quasilinear parabolic system in Section \ref{numericalresults}.
\begin{figure}[thb!]
\centering
 \begin{subfigure}[b]{0.48\textwidth}
\includegraphics[width=\textwidth]{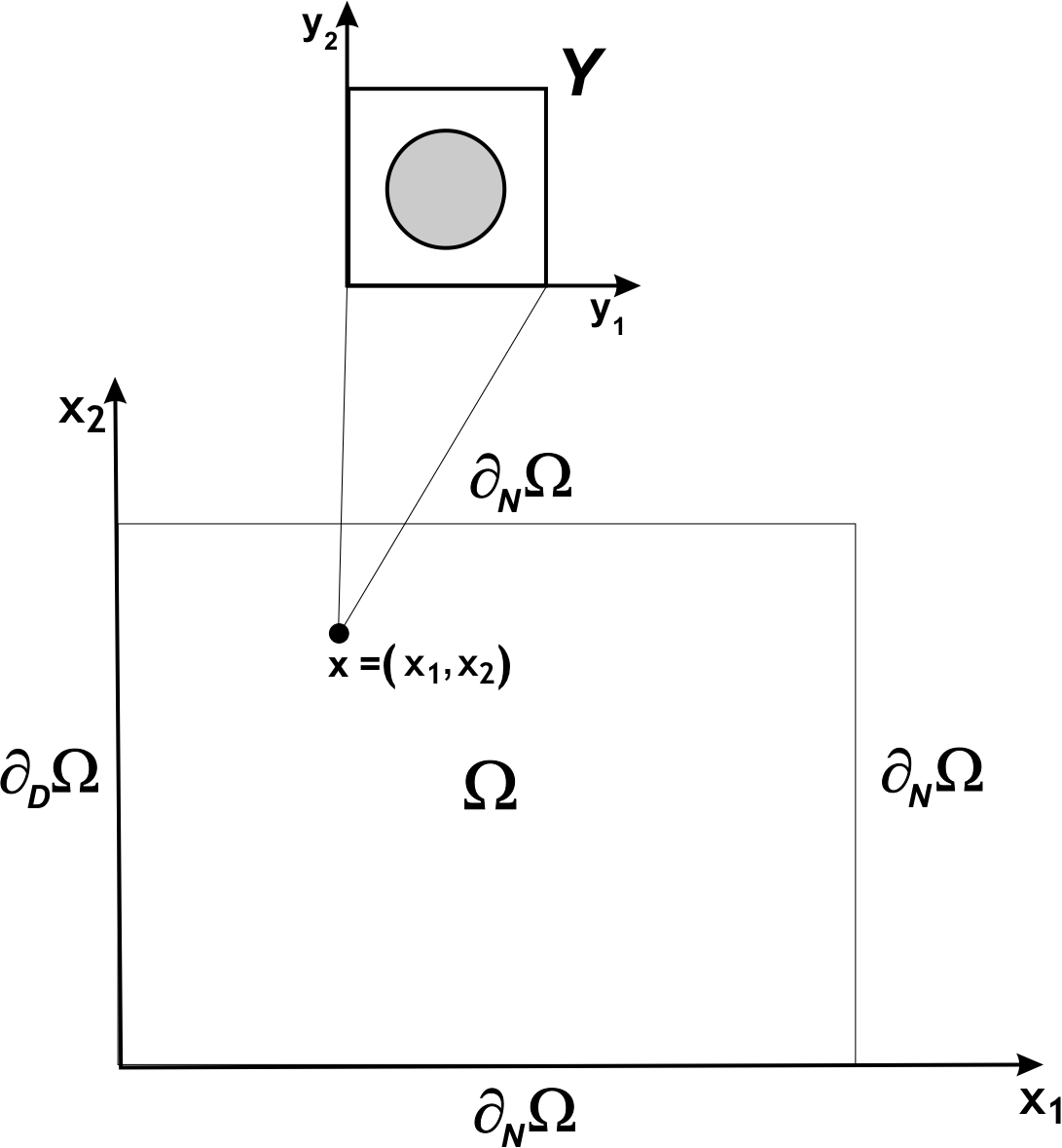}
\caption{}
\end{subfigure}
~
 \begin{subfigure}[b]{0.48\textwidth}
\includegraphics[scale=0.65]{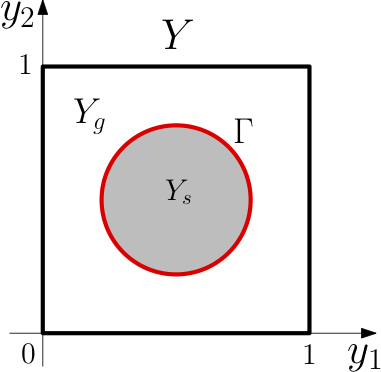}
\caption{}
\end{subfigure}
\caption{Schematic of the macroscopic domain $\Om \subset \mathbb{R}^2$ (a) showing its coupling with a reference unit cell $Y$ for each point $x\in \Om$ and description of the boundary conditions at different parts of the exterior boundary $\p\Om$. The microstructure (b) consists of a solid part $Y_{\rm s}$ and gas part $Y_{\rm g}$ which are separated on by an interface $\Gamma$.}
\label{Fig1}
\end{figure}
\section{Setting of the problem}\label{sec:setting}
We consider a homogenization limit problem of a filtration combustion process derived under the assumption of large Damkh\"oler numbers and small P\'eclet numbers, which results in a system of quasilinear diffusion equations modeling heat and mass diffusion processes. These equations are coupled by nonlocal reactions taking place on a gas-solid interface separating the reactant species at the level of the microstructure of the porous domain. The reactions are incorporated into the global diffusion coefficients and account for a reaction-diffusion phenomena at the macroscopic level. In the sequel, we describe the multiscale geometry of the studied problem.
\subsection{The geometry}\label{eqn:boundaryconditions}
Let $\Om \subset \Real^d, d=2, 3,$ be open and bounded homogeneous domain that approximates the heterogeneous porous medium consisting of a periodic system of fixed microstructures. In our setting, $\Om$ is either a polygon for $d=2$ or a polyhedron for $d=3$ with Lipschitz boundary $\p \Om=\p_{D}\Om \cup \p_{N}\Om$ which is either an edge for $d=2$ or a face for $d=3$, where
\begin{equation}
 \p_{D}\Om :=\p\Om \cap \{x :=(x_1, x_2, \ldots,x_d) \in \Real^d \mid x_1=0\} ~\mbox{  and  }~ \p_{N}\Om :=\p\Om \setminus \p_{D}\Om,
\end{equation}
are the Dirichlet and Neumann boundary parts, respectively. For each point $x \in \Om$, let the representative cell be denoted by $Y=[0,1]^d.$ The representative cell consists of two distinct parts-a solid part denoted by $Y_{\rm s}$ and gas-filled part denoted by $Y_{\rm g}$, i.e. \mbox{$Y:=Y_{\rm s} \cup Y_{\rm g}$.} We denote the smooth boundary of the solid part of $Y$ by $\Gamma.$ The gas-filled part is denoted by \mbox{$Y_{\rm g}=Y\setminus \overline{Y_{\rm s}}$}. On the gas-solid boundary $\Gamma$, interface conditions are prescribed whereas  periodic boundary conditions are prescribed on the exterior boundary $\p Y$ of $Y$.  The outer unit normal to the boundaries of the domains is denoted by $\vec n$.
Furthermore, let $T>0$ be an arbitrarily chosen final time such that we can use the following time-space domains:
\begin{align*}
&\Om^T = {(0,T]}\times \Om, \quad \p_{D}\Om^{T} =  {(0,T]}\times \p_{D}\Om, \quad  \p_{N}\Om^{T} =  {(0,T]}\times \p_{N}\Om.
\end{align*}
\subsection{The mathematical model}
\noindent Let the scaled temperature of the gas-solid system and the scaled concentration of the gaseous oxidizer be given respectively by $u(t,x)$ and $v(t,x)$, then the heat and mass diffusion equations can be written as
\begin{align}
\label{heatmassdiffusion}
\begin{dcases}
c\dfrac{\p u}{\p t} = \Na\!\cdot\Big(\La(u,v)\Na u\Big), &\text{in} \quad \Om^T,\\
\theta\dfrac{\p v}{\p t} = \Na\!\cdot\Big(\D(u,v)\Na v\Big), &\text{in} \quad  \Om^T,\\
u = u_{D}, v= v_{D},&\text{on} \quad \p_D\Om^T,\\
\vec n\!\cdot\Na u = \vec n\!\cdot\Na v = 0,&\text{on} \quad \p_N\Om^T,\\
u(t,x) = u_{I}, v(t,x)= v_{I}, &\text{in}  \quad \overline{\Om}, t=0.
\end{dcases}
\end{align}
{\allowdisplaybreaks
Here, $u_{I}$ and $v_{I}$ are the prescribed initial data and $u_{D}$ and $v_{D}$ are the Dirichlet boundary data on $\p_{D}\Om^{T}$ and $\vec n$ is the outward unit normal. The effective parameters appearing in the equations are the porosity $\theta$ and the effective volumetric heat capacity $c$ respectively given by:
\begin{align*}
\theta =\dfrac{1}{|Y|}\int\limits_{Y_{\rm g}}dy \quad \text{and} \quad 
c  = \int\limits_{Y_{\rm g}} c_{\rm g}dy + \int\limits_{Y_{\rm s}}c_{\rm s}dy,
\end{align*}
where $c_{\rm g}$ and $c_{\rm s}$ are respectively the heat capacities for gas and solid parts of the microstructure. The diffusion matrices are given by \mbox{$\La_{i,j}:\Real\times\Real \rightarrow \Real$}
\begin{align} 
\label{Lmatrix}
\La_{ij}(s,r) & =\int\limits_{Y_{\textrm g}}\La_{\textrm g}(y)\big(e_i + \Na_y\chi_{\textrm{g},i}\big)\!\cdot\!\big(e_j + \Na_y\chi_{\textrm{g},j}\big)dy
\notag \\ 
& +\int\limits_{Y_{\textrm s}}\La_{\textrm s}(y)\big(e_i + \Na_y\chi_{\textrm{s},i}\big)\!\cdot\!\big(e_j + \Na_y\chi_{\textrm{s},j}\big)dy\\\nonumber
&+Qf(s)\int\limits_{\Gamma}\Big[r\chi_i\chi_j + \dfrac{s^2}{2u_{\textrm a}}\big(\omega_j\chi_i + \omega_i\chi_j\big)\Big]d\sigma,\quad i,j=1,\ldots,d, \mbox{for all $s,r \in \Real$}
\end{align}
and 
\begin{align}
\label{Dmatrix}
\D_{ij}(s,r) & = \int\limits_{Y_{\textrm g}}D(y)\big(e_i + \Na_y\omega_{i}\big)\!\cdot\!\big(e_j + \Na_y\omega_{j}\big)dy \notag \\
& -f(s)\int\limits_{\Gamma}\Big[\dfrac{s^2}{u_{\textrm{a}}}\omega_i\omega_j + \dfrac{r}{2}\big(\omega_i\chi_j + \omega_j\chi_i\big)\Big]d\sigma,\quad i,j=1,\ldots,d, \mbox{for all $s,r \in \Real,$}
\end{align}
where $(\chi,\omega) = \big(\chi_j,\omega_j\big)_{j=1,\ldots,d}$ is the solution of the coupled cell problem 
\begin{align}
\label{cellproblem}
\begin{dcases}
-\Na_y\!\cdot\!(\La_{\textrm g}(y)(\Na_y\chi_{\textrm{g},j} + e_j) )=0,&\mbox{in $\Om \times Y_{\textrm g},$}\\
-\Na_y\!\cdot\!(\La_{\textrm s}(y)(\Na_y\chi_{\textrm{s},j} + e_j)) =0,&\mbox{in $\Om\times Y_{\textrm s},$}\\
\chi_{\textrm{g},j}-\chi_{\textrm{s},j}=0,&\mbox{on $\Om\times\Gamma,$}\\
\big[\La_{\textrm s}(y)(\Na_y\chi_{\textrm{s},j}+e_j)-\La_{\textrm g}(y)(\Na_y\chi_{\textrm{g},j} + e_j)\big]\!\cdot\!\vec n=Qf(u)\cH(u,v,\chi_j,\omega_j),&\mbox{on $\Om\times \Gamma,$}\\
-\Na_y\!\cdot\!(D(y)(\Na_y\omega_{j} + e_j) ) = 0,&\mbox{in $\Om\times Y_{\textrm g},$}\\
-D(y)(\Na_y\omega_{j} + e_j)\!\cdot\!\vec n=-f(u)\cH(u,v,\chi_j,\omega_j),&\mbox{on $\Om\times\Gamma,$}\\
y \rightarrow (\chi(y),\omega(y)) \mbox{  is $Y$-periodic,}
\end{dcases}
\end{align}
where $u$ and $v$ are solutions of \eqref{heatmassdiffusion}. In \eqref{cellproblem}, $\La_{\rm g}$ and  $\La_{\rm s}$  are respectively the heat conductivities for gas and solid parts of the microstructure, $Q>0$ is the heat release, $u_{\rm a}$ is the activation temperature and $A$ is the pre-exponential factor in the Arrhenius kinetics. We assume all physical quantities to be constants and dimensionless. Furthermore, the natural choices for the nonlinear terms $f$ and $\cH$ are those deduced from the homogenization procedure; see, e.g. \cite{Ijioma2016}. In particular, we adopt the following form in our calculations:
\begin{align}
\label{Arrhenius}
f(s) = \dfrac{Au_{\rm a}}{s^2}\exp\Big(-\dfrac{u_{\rm a}}{s}\Big) \quad \text{and} \quad 
\cH(s,r,\varphi,\psi) = r\varphi + \dfrac{s^2}{u_{\rm a}}\psi.
\end{align}
\subsection{Working hypothesis }
\noindent We assume the following for the effective tensors: 
\begin{itemize}
\item[($H1$)] The effective diffusion tensors are continuous and satisfy the following
\begin{align*}
&0 < \La_0 \leq \La(s,r) \leq \La_1,\qquad 0 < D_0 \leq \D(s,r) \leq D_1\\
&\abs{\p_s\La(s,r)} + \abs{\p_r\La(s,r)} \leq \beta_0, \quad \abs{\p_s\D(s,r)} + \abs{\p_r\D(s,r)} \leq \beta_1,
\end{align*}
for all $s, r\in \Real$ and for $\La_i,D_i,\beta_i\in\Real, i=0,1.$
\end{itemize}
For the physical properties of the material restricted to the pore domain, we assume:
\begin{itemize}
\item[($H2$)] The molecular diffusion is periodic, isotropic and restricted to the gas region, i.e., $D\in L^{\infty}_{\#}(Y_{\rm g})$. In addition, it satisfies the uniformly coercive property, i.e., there exists a constant $C>0$ such that, for any $\xi \in \Real^d$,
$$D(y)\xi\cdot\xi \ge C|\xi|^2 \mbox{ a.e. $y\in Y_{\rm g}.$ }$$
Similarly, the thermal conductivity is periodic, isotropic and varies in the gas and solid regions, i.e., $\lambda \in L^{\infty}_{\#}(Y)$ and
\begin{align}
\lambda(y) = 
\begin{dcases}
\lambda_{\rm g}, &\mbox{in $Y_{\rm g}$,}\\
\lambda_{\rm s}, &\mbox{in $Y_{\rm s}.$}
\end{dcases}
\end{align}
In addition, it satisfies the uniformly coercive property, i.e., there exists a constant $C>0$ such that, for any $\xi \in \Real^d$,
$$\lambda(y)\xi\cdot\xi \ge C|\xi|^2 \mbox{ a.e. in $Y.$ }$$
The heat capacity is periodic and it is defined in the gas and the solid regions, i.e.,  $c \in L^{\infty}_{\#}(Y)$ is such that $\hat{c} \leq c \leq \tilde{c}$ and
\begin{align}
c(y) = 
\begin{dcases}
c_{\rm g}, &\mbox{in $Y_{\rm g}$,}\\
c_{\rm s}, &\mbox{in $Y_{\rm s}.$}
\end{dcases}
\end{align}
\end{itemize}
For the reaction terms, we assume
\begin{itemize}
\item[($H3$)] Let $f: \Real\rightarrow \Real$ be such that, for appropriate choice of the parameter $u_{a} \gg 0$, $f(s) \leq As$  and hence globally Lipschitz for all $s\in \Real$ and $A>0$.
\item[($H4$)] $\cH: \Real\times\Real\times\Real\times\Real \rightarrow \Real$ is Lipschitz continuous with respect to the variables. 
\end{itemize}
For the initial and boundary functions, we assume
\begin{itemize}
\item[($H5$)] $u_{I},v_{I}\in H^2(\Om)\cap L^{\infty}_{+}(\Om)$
\item[($H6$)] $u_{D}, v_{D}\in L^{2}(0,T;H^2(\Om))\cap L^{\infty}_{+}(\Om^{T}).$
\end{itemize}
\subsection{Weak formulation}
\noindent In order to adapt the Galerkin approach to the approximation of the unique weak solution of the system \eqref{heatmassdiffusion}-\eqref{cellproblem}, we make use of the following function spaces in the weak formulation of the problem:
\begin{align*}
&H^1_D(\Om) = \big\{~\varphi\in H^1(\Om) \mid \mbox{$\varphi= 0$ on $\p_D\Om$}~\big\},\\
&H^1_{\#}(Y) =\big\{~\varphi\in H^1(Y) \mid \mbox{$\varphi$ is $Y$-periodic, $\int\limits_{Y} \varphi =0 \,$}~\big\}.
\end{align*}
\begin{definition}
\label{Weak}
A quadruple of functions $(u,v,\chi,\omega)$ with 
\begin{align*}
&(u-u_D)\in L^2((0,T);H^1_D(\Om)),(v-v_D)\in L^2((0,T);H^1_D(\Om)),\p_t u \in L^2(0,T;V), \\
&\p_t v \in L^2(0,T;V),\chi \in L^2(\Om;H^1_{\#}(Y)^d), \omega \in L^2(\Om;H^1_{\#}(Y_{\rm g})^d)
\end{align*}
is called a weak solution of \eqref{heatmassdiffusion}-\eqref{cellproblem} if for a.e. $t\in (0,T)$ the following identities hold
\begin{align}
\label{varfHomo}
&\int\limits_{\Om}c\p_t u\varphi dx + \int\limits_{\Om}\La(u,v)\Na u\Na\varphi dx+\int\limits_{\Om}\theta\p_t v\vartheta dx +  \int\limits_{\Om}\D(u,v)\Na v\Na\vartheta dx=0,\\
\label{varfCell} 
&\int\limits_{\Om\times Y_g}\La_{g}\big(e_{j} + \Na_{y}\chi_{g,j}\big)\cdot\Na_{y}\phi dydx + \int\limits_{\Om\times Y_s}\La_{s}\big(e_{j} + \Na_{y}\chi_{s,j}\big)\cdot\Na_{y}\phi dydx \notag \\
&+ Q\int\limits_{\Om\times Y_g}D(y)\big(e_{j} + \Na_{y}\omega_{j}\big)\cdot\Na_{y}\psi dydx 
+ Q\int\limits_{\Om}\int\limits_{\Gamma}f(u)\cH(u,v,\chi_j,\omega_j)(\phi-\psi) d\sigma(y)dx=0
\end{align}
for all $(\varphi,\vartheta,\phi,\psi)\in H^1_D(\Om)^2\times L^2(\Om;H^1_{\#}(Y)^d)\times L^2(\Om;H^1_{\#}(Y_{\rm g})^d)$ with $d=2, 3$ and
\begin{eqnarray}
u(0)=u_{I}~\mbox{in $\Om$},~v(0)=v_{I}~\mbox{in $\Om.$}
\end{eqnarray}
\end{definition}
\begin{lemma}
\label{UCell}
For any given values of $s, r\in \Real_{+}$, there exists a unique solution
\[
(\chi,\omega)=(\chi_j,\omega_j)_{j=1,\ldots,d}\in [H^1_{\#}(Y)]^d\times [H^1_{\#}(Y_{\rm g})]^d,
\]
to the cell problem \eqref{cellproblem} up to the addition of a constant multiple of ($C,C$) with $C\in\Real$. 
\end{lemma}
\begin{proof}
It can be shown that the variational formulation \eqref{varfCell} satisfies the assumptions of the Lax-Milgram Lemma. Let the quotient space $[H^1_{\#}(Y)\times H^1_{\#}(Y_{\rm g})]/\Real(C,C)$ of functions defined in $H^1_{\#}(Y)\times H^1_{\#}(Y_{\rm g})$ up to an additive constant vector $(C,C)$ with $C\in\Real$ be associated with the cell solutions $(\chi,\omega)$. It is easily seen that $||\Na \chi ||_{L^2(Y)^d} + ||\Na\omega ||_{L^2(Y_{\rm g})^d}$ is a norm for this space. By $(H2)$, it can be shown that the left hand side of \eqref{varfCell} is coercive on the quotient space. Again, using $(H2)$, the right hand side of \eqref{varfCell}
\begin{align}
\int\limits_{Y}\La(y)e_{j}\Na_{y}\phi dy + Q\int\limits_{Y_{\rm g}}D(y)e_{j}\Na_{y}\phi dy \quad\mbox{for a.e. $x\in \Om,$}
\end{align}
is a continuous linear functional on the quotient space.
\end{proof}
\begin{lemma}\label{coercive}
Assume $(H3)$--$(H4)$, the diffusion tensors $\La(s,r)$ and $\D(s,r)$ are uniformly coercive and bounded in the sense of $(H1)$-$(H2)$.
\end{lemma}
\begin{proof}
Since $f(s)\leq \alpha$ implies that $\D(s,r)\leq \int\limits_{Y_{\rm g}}D(y)dy$, and hence $\D(s,r)$ and $\La(s,r)$ are uniformly bounded.~Given that the diffusion tensors \eqref{Lmatrix} and \eqref{Dmatrix} are symmetric and $f(s)\geq 0$, we have that 
\begin{eqnarray}
\label{Llbound}
\La(s,r) \geq \Bigg[\int\limits_{Y_{\rm g}}\La_{\rm g}(y)~dy + \int\limits_{Y_{\rm s}}\La_{\rm s}(y)~dy \Bigg]>\La_0,
\end{eqnarray}
for some positive constant $\La_0\in\Real.$ Furthermore, by multiplying \eqref{Dmatrix} by $Q>0$ and adding the resulting expression to \eqref{Lmatrix} yields a lower bound given by the right hand side of \eqref{Llbound}. Hence, $\D(s,r)$ is bounded from below by
 \begin{eqnarray}
 \label{Dlbound}
\D(s,r) \geq \dfrac{1}{Q}\Bigg[\int\limits_{Y_{\rm g}}\La_{\rm g}(y)~dy + \int\limits_{Y_{\rm s}}\La_{\rm s}(y)~dy-\La_0\Bigg].
\end{eqnarray}
Thus, $\La(s,r)$ and $\D(s,r)$ are uniformly coercive.
\end{proof}
\subsection{Properties of $f$}
Since the function $f$\footnote{The primitive of the function $f$ is the Arrhenius kinetics in its standard form; see, e,g, \cite{Ijioma13}.}  defined by \eqref{Arrhenius} is undefined at $x=0,$ we complete its definition such that $f$ is continuous on whole $\Real$ by rewriting \eqref{Arrhenius} as:  
\begin{align}\label{Eqf1}
f(s)=
\begin{dcases}
\dfrac{Au_a}{s^2}\exp\bigg(-\dfrac{u_a}{s}\bigg),&\quad\mbox{$s > 0$}\\
0,&\quad\mbox{$s \leq 0.$}
\end{dcases}
\end{align}  
In \eqref{Eqf1}, $A>0$ is the nondimensional pre-exponential factor and $u_a \gg 0$ is the nondimensional activation temperature. The continuity of \eqref{Eqf1} can be seen by applying the $L^\prime$H\^opital's rule to the right hand limit of $f$ as $s\rightarrow 0$, noting that the left hand side limit is zero. Furthermore, the condition on $u_a$ is motivated physically since the activation energy of solid fuels is usually large; see, e.g. \cite{KASHIWAGI92}. Thus, \eqref{Eqf1}, for appropriate choices of $u_a,$ has at most a linear growth in $s$, i.e.
\begin{align}\label{Eqf2}
f(s) \leq As.
\end{align}
From \eqref{Eqf2}, it can easily be shown that $f$ is globally Lipschitz continuous in $s$. Alternatively, the derivative of $f$ defined by
\begin{align}\label{Eqf3}
f^\prime(s)=
\begin{dcases}
-\dfrac{2Au_a}{s^3}\exp\bigg(-\dfrac{u_a}{s}\bigg) + \dfrac{Au_a^2}{s^4}\exp\bigg(-\dfrac{u_a}{s}\bigg),&\quad\mbox{$s > 0$}\\
0,&\quad\mbox{$s \leq 0,$}
\end{dcases}
\end{align}
is continuously differentiable and bounded since $$\lim_{s^-\rightarrow 0} f^\prime(s)=\lim_{s^+\rightarrow 0}  f^\prime(s) =0,\mbox{ and there exists a constant $C$ such that}$$ \mbox{$|f^\prime(s)| \leq C$ for all $s\in \Real$}. Hence, $f$ is a Lipschitz function. It is worth mentioning that since \eqref{Eqf3} is of the same form as \eqref{Eqf1}, it holds also that $f^\prime(s) \leq As, \mbox{ for all $s\in \Real$.}$
\begin{lemma}\label{boundedTensor}
Assume $(H3)$. Then, for all $s, r\in \Real$ there exists a constant C such that
\begin{align}\label{boundedmatrix}
&|\La_s(s,r)| + |\La_r(s,r)| \leq C,\\
&|\D_s(s,r)| + |\D_r(s,r)| \leq C.
\end{align}
\end{lemma}
\begin{proof}
From formulas \eqref{Lmatrix} and \eqref{Dmatrix} of the tensors, we only need to show that the derivatives with respect to $s$ and $r$ of the surface integrals:
\begin{align}\label{Eqf4}
\La^S_{ij} = Qf(s) \int\limits_{\Gamma}\bigg[r\chi_i\chi_j + \dfrac{s^2}{2u_a}(\omega_j\chi_i + \omega_i\chi_j)\bigg]d\sigma,\\
\D^S_{ij} = -f(s) \int\limits_{\Gamma}\bigg[\dfrac{s^2}{u_a}\chi_i\chi_j + \dfrac{r}{2}(\omega_j\chi_i + \omega_i\chi_j)\bigg]d\sigma
\end{align}
are bounded, where $f$ is given by \eqref{Eqf1}. Rewriting \eqref{Eqf4} for $\La^S_{ij}$ leads to
\begin{align}\label{Eqf5}
\La^S_{ij}  = S_1rf(s) + S_2\exp\bigg(\dfrac{-u_a}{s}\bigg),
\end{align}
where 
\begin{eqnarray}\label{Eqf5a}
S_1=Q\int\limits_{\Gamma}\chi_i\chi_j~d\sigma,\quad S_2=\dfrac{Q}{2}\int\limits_{\Gamma}(\omega_j\chi_i + \omega_i\chi_j)~d\sigma.
\end{eqnarray}
The integral coefficients \eqref{Eqf5a} are bounded since by the interpolation trace inequality \cite{LadyzhenskayaSolonnikovUralceva:1968a}, we get
\begin{align}
S_1 = \int\limits_{\Gamma}\chi_i\chi_j~d\sigma &\leq C \int\limits_{\Gamma}(|\chi_i|^2 + |\chi_j|^2)~d\sigma\\\nonumber
& \leq C\Big(\|\chi_i\|_{L^2(Y)}\|\chi_i\|_{H_\#^1(Y)} + \|\chi_j\|_{L^2(Y)}\|\chi_j\|_{H_\#^1(Y)}\Big)\\\nonumber
& \leq C\Big(\|\chi_i\|^2_{H_\#^1(Y)} + \|\chi_j\|^2_{H_\#^1(Y)}\Big).
\end{align}
Then, by Lemma \ref{UCell}, there exists a bounded solution $\chi \in H_\#^1(Y)$, unique up to an addition of a constant $C\in \Real.$ It is easy to see that a similar argument can be arrived at for $S_2$. Now, differentiating \eqref{Eqf5} with respect to $s$ yields
\begin{align}\label{Eqf6}
|\La^S_{ij,s}| \leq |S_1||r||f^\prime(s)| + |S_2| |f(s)|.
\end{align}
Using $(H3)$ (property \eqref{Eqf2} of $f$ and $f^\prime$) and the fact that $|S_1|$ and $|S_2|$ are bounded integrals since the functions, $\chi$ and $\omega$, are bounded, then the Lipschitz criteria follows for any $s_1,s_2\in \Real$ and $r\in \Real$, it holds that
\begin{eqnarray}\label{Eqf7}
|\La^S_{ij,s}| \leq C|s_1 - s_2|\leq C(|s_1| + |s_2|).
\end{eqnarray}
Similarly, taking the derivative of \eqref{Eqf5} with respect to $r$ leads to $S_1f(s).$ Since $f$ and $S_1$ are bounded from above, it follows that
\begin{eqnarray}\label{Eqf8}
|\La^S_{ij,r}| \leq C(|r_1| + |r_2|), \quad\mbox{for all $r_1,r_2\in\Real.$}
\end{eqnarray}
By summing \eqref{Eqf7} and \eqref{Eqf8}, we arrive at $\eqref{boundedmatrix}_1$. Repeating the same steps as above, it is easily seen that
\begin{align}\label{Eqf9}
|\D^S_{ij,s}| +|\D^S_{ij,r}| \leq C,\quad\mbox{for all $s,r\in \Real.$}\\\nonumber
\end{align}
\end{proof}
The uniqueness of the quasilinear parabolic system of equations is shown in the following proposition.
\begin{proposition}
Assume $(H1)$ and $(H3)$ hold. Then, there exists a unique weak solution to \eqref{heatmassdiffusion}.
\end{proposition}
\begin{proof}
Let $u_1, u_2 ,v_1,v_2$ be arbitrary weak solutions of \eqref{heatmassdiffusion}. We choose $(\varphi, \vartheta)=(u_2-u_1, v_2-v_1)\in H^1_D(\Om)\times H^1_D(\Om)$ as test function in the variational formulation \eqref{varfHomo}. We obtain 
\begin{align}\label{Uniq1}
&\dfrac{1}{2}\dfrac{d}{dt}\int\limits_{\Om}c|u_2-u_1|^2~dx + \dfrac{1}{2}\dfrac{d}{dt}\int\limits_{\Om}\theta |v_2-v_1|^2~dx + \int\limits_{\Om}\Big(\La(u_2,v)-\La(u_1,v)\Big)|\Na(u_2-u_1)|^2~dx \nonumber\\
&+\int\limits_{\Om}\Big(\D(u,v_2)-\D(u,v_1)\Big)|\Na(v_2-v_1)|^2~dx =0.
\end{align} 
In \eqref{Uniq1}, we require to show that $\La(s,r)$ and $\D(s,r)$) are bounded. We show boundedness for $\La$, while $\D$ follows a similar argument. By the fundamental theorem of calculus and Lemma \ref{boundedTensor}, it is easy to see that
\begin{align}\label{Uniq2}
|\La(s_2,r)-\La(s_1, r)| &= \int\limits_{s_1}^{s_2}|\La_s(s,r)|~ds \leq C|s_2-s_1|,\quad\mbox{$r\in\Real$}.
\end{align} 
Substituting \eqref{Uniq2} in the third and last integral of \eqref{Uniq1}, we get
\begin{align}\label{Uniq3}
&\int\limits_{\Om}\Big(\La(u_2,v)-\La(u_1,v)\Big)|\Na(u_2-u_1)|^2~dx  + \int\limits_{\Om}\Big(\D(u, v_2)-\D(u,v_1)\Big)|\Na(v_2-v_1)|^2~dx \nonumber\\
& \leq C\int\limits_{\Om}|u_2-u_1||\Na(u_2-u_1)|^2~dx + C\int\limits_{\Om}|v_2-v_1|\Na(v_2-v_1)|^2~dx \nonumber\\
& \leq \|u_2-u_1\|^2_{L^2(\Om)} + \|v_2-v_1\|^2_{L^2(\Om)} + \|\Na (u_2-u_1)\|^2_{L^2(\Om)} +  \|\Na (v_2-v_1)\|^2_{L^2(\Om)}.
\end{align}
By substituting \eqref{Uniq3} in \eqref{Uniq1}, integrating with respect to time and applying the Gronwall's inequality, we deduce the desired result.
\end{proof}
\subsection{Main result}
\noindent The main result of this paper is summarized in the following theorem
\begin{theorem}
Let the assumptions $(H1)$-$(H6)$ be satisfied. Assume further that the projection operators $P^N_x, P^M_y$ defined in \eqref{proj1} and \eqref{proj3} are stable with respect to the $L^2$-norm and $H^2$-norm. Let $(u^N_0,v^N_0,\chi^{N,M},\omega^{N,M})$ be the finite-dimensional approximations defined in \eqref{uN}-\eqref{wN}. Then, for $N,M\rightarrow \infty$, the sequence $(u^N_0 + u_D,v^N_0+v_D,\chi^{N,M},\omega^{N,M})$ converges to the unique weak solution $(u,v,\chi,\omega)$ of problem \eqref{heatmassdiffusion}-\eqref{cellproblem}.
\end{theorem}
\section{Global existence of weak solutions}\label{galerkinapprox}
To show global existence of weak solutions to problem \eqref{heatmassdiffusion}--\eqref{Arrhenius}, we use a multiscale Galerkin method to exploit the two-scale nature of the problem, which results to defining finite dimensional approximations for the solutions of \eqref{heatmassdiffusion}--\eqref{Arrhenius}. A key aspect to defining the finite dimensional approximation is the choice of the bases. For this purpose, we take clues from \cite{MunteanRadu:2010}. The basis elements on the domain $\Om\times Y$ are chosen as tensor products of basis elements on the macroscopic domain $\Om$ and on the representative cell $Y$.

The prove of convergence of the finite-dimensional approximations to the weak solution of problem \eqref{heatmassdiffusion}--\eqref{Arrhenius} is determined by the uniform estimates proved in Subsection \ref{subsect:uniformestimatesforthediscretizedproblems}. The convergence step is analogous to standard Galerkin approximation, however, compactness results for the finite-dimensional approximations is required for both the microscopic and macroscopic variables. The main difficulty is in handling the nature of the coupling between macroscopic variables and the cell variables and vice versa; the coupling exhibited in the multiscale scenario is such that the microscopic variables are used in the calculation of the diffusion tensors whereas the macroscopic variables enter the cell problems as parameters to the interface conditions on $\Gamma$.
\subsection{Galerkin approximation and global existence for the discretized problem}
\noindent Let $\{\xi_i\}_{i\in  \mathbb{N}}$ be a basis of $L^2(\Om)$, with $\xi_j\in H^2(\Om) \cap H^1_{D}(\Om)$, forming an orthonormal system ($\rm ONS$) with respect to $L^2(\Om)$-norm. Also, let $\{\zeta_{jk}\}_{j,k\in\mathbb{N}}$ be a basis of $L^2(\Om\times Y),$ with
{\allowdisplaybreaks
\begin{eqnarray}
\zeta_{jk}(x,y)=\xi_j(x)\eta_k(y),
\end{eqnarray}
where $\{\eta_k\}_{k\in \mathbb{N}}$ is a basis of $L^2(Y),$ with $\eta_k\in H^2(Y) \cap H^1_{\#}(Y)$, forming an $\rm ONS$ with respect to $L^2(Y)$-norm. We define the projection operators on finite dimensional subspaces $P^N_x, P^M_y$ associated with the bases $\{\xi_j\}_{j\in \mathbb{N}}$ and $\{\eta_k\}_{k\in  \mathbb{N}}$ respectively. For $(\varphi,\psi)$ of the form
\begin{align}
&\varphi(x) = \sum_{j\in \mathbb{N}}a_j\xi_j(x),\\
&\psi(x,y)=\sum_{j,k\in \mathbb{N}}b_{jk}\xi_j(x)\eta_k(y),
\end{align}
we define
\begin{align}
\label{proj1}
&\Big(P^N_x\varphi\Big)(x) = \sum_{j=1}^{N}a_j\xi_j(x),
\quad \Big(P^N_x\psi\Big)(x,y) = \sum_{j=1}^{N}\sum_{k\in \mathbb{N}}b_{jk}\xi_j(x)\eta_k(y),\\
\label{proj3}
&\Big(P^M_y\psi\Big)(x,y) = \sum_{j\in \mathbb{N}}\sum_{k=1}^M b_{jk}\xi_j(x)\eta_k(y).
\end{align}
The bases $\{\xi_j\}_{j\in \mathbb{N}}$ and $\{\eta_k\}_{k\in  \mathbb{N}}$ are chosen such that the projection operators $P^N_x,P^M_y$ are stable with respect to the $L^2$-norm and $H^2$-norm; i.e. for a given function $\varphi$, the $L^2$-norm and $H^2$-norm of the truncations by the projection operators can be estimated by the corresponding norms of the function. In the next step, we look for finite-dimensional approximations of the functions 
\begin{eqnarray*} u_0 =u-u_D, v_0 =v-v_D, \chi ~\mbox{and}~\omega \end{eqnarray*} 
of the following form:
\begin{align}
\label{uN}
&u_0^N(t,x) = \sum_{j=1}^{N}\U^N_j(t)\xi^{\rm u}_j(x),\\
\label{vN}
&v_0^N(t,x) = \sum_{j=1}^N\V^N_j(t)\xi^{\rm v}_j(x),\\
\label{xN}
&\chi^{N,M}_{\alpha,p}(x,y) = \sum_{j=1}^N\sum_{k=1}^M\X_{jk,p}\zeta_{jk}^{\chi}(x,y),~\mbox{$\alpha=\{g,s\},p=1,\ldots,d,$}\\
\label{wN}
&\omega_p^{N,M}(x,y)=\sum_{j=1}^N\sum_{k=1}^M\W_{jk,p}\zeta^{\omega}_{jk}(x,y),~\mbox{$p=1,\ldots,d,$}
\end{align}
where the coefficients $\U^N_j,\V^N_j,\X_{jk,p},\W_{jk,p}, j=1,\ldots,N, k=1,\ldots,M, p=1,\ldots,d$ are determined by the following relations:
\begin{align}
\label{Disc1}
&\int\limits_{\Om}c\pt u_0^N\varphi dx + \int\limits_{\Om}\La\big(u^N_{0} + u_D,v^N_{0} + v_D\big)\Na u^N_0\Na\varphi dx + \int\limits_{\Om}\theta \pt v^N_0 \vartheta dx \notag \\ 
& + \int\limits_{\Om}\D\big(u^N_{0} + u_D,v^N_{0} + v_D\big)\Na v^N_0\Na \vartheta dx \nonumber \\
&=\int\limits_{\Om}\Big(\Na\cdot\Big(\La\big(u^N_{0} + u_D,v^N_{0} + v_D\big)\Na u_D(t)\Big) - c\pt u_D(t)\Big)\varphi dx \notag \\
& + \int\limits_{\Om}\Big(\Na\cdot\Big(\D\big(u^N_{0} + u_D,v^N_{0} + v_D\big)\Na v_D(t)\Big) - \theta\pt v_D(t)\Big)\vartheta dx,
\end{align}
and 
\begin{align}
\label{Disc2}
&\int\limits_{\Om\times Y_{\rm g}}\La_{\rm g}\big(e_j + \Na_y \chi^{N,M}_{\textrm g,j} \big)\cdot\Na_y\phi dydx + \int\limits_{\Om\times Y_{\rm s}}\La_{\rm s}\big(e_j + \Na_y \chi^{N,M}_{\textrm s,j} \big)\cdot\Na_y\phi \,dydx \notag \\
& + Q\int\limits_{\Om\times Y_{\rm g}}D(y)\big(e_j + \Na_y\omega^{N,M}_j\big)\cdot\Na_y\psi dydx \\
&+Q\int\limits_{\Om}\int\limits_{\Gamma}f(u^N_{0} + u_D)\cH(u^N_0+u_D,v^N_0+v_D,\chi^{N,M}_j,\omega^{N,M}_j)\big(\phi-\psi\big)\, d\sigma dx=0 \notag 
\end{align}
with 
\begin{align}
\label{Disc1-lambda}
&\La_{ij}(u^N_{0} + u_D,v^N_{0} + v_D)  = \int\limits_{Y_{\textrm g}}\La_{\textrm g}(y)\big(e_i + \Na_y\chi^{N,M}_{\textrm{g},i}\big)\!\cdot\!\big(e_j + \Na_y\chi^{N,M}_{\textrm{g},j}\big)dy \\\nonumber
& +\int\limits_{Y_{\textrm s}}\La_{\textrm s}(y)\big(e_i + \Na_y\chi^{N,M}_{\textrm{s},i}\big)\!\cdot\!\big(e_j + \Na_y\chi^{N,M}_{\textrm{s},j}\big)dy\\\notag
&+Qf(u^N_{0} + u_D)\int\limits_{\Gamma}\Big[\Big(v^N_{0} + v_D\Big)\chi^{N,M}_i\chi^{N,M}_j 
+ \dfrac{(u^N_{0} + u_D)^2}{2u_{\textrm a}}\Big(\omega^{N,M}_j\chi^{N,M}_i + \omega^{N,M}_i\chi^{N,M}_j\Big)\Big]d\sigma,\\\notag
&\quad\mbox{for $i,j=1,\ldots,d$}
\end{align}
and
\begin{align}
\label{Disc2-D}
 &\D_{ij}(u^N_{0} + u_D,v^N_{0} + v_D) = \int\limits_{Y_{\textrm g}}D(y)\big(e_i + \Na_y\omega^{N,M}_{i}\big)\!\cdot\!\big(e_j + \Na_y\omega^{N,M}_{j}\big)dy \nonumber \\
&-f(u^N_{0} + u_D)\int\limits_{\Gamma}\Big[\dfrac{(u^N_{0} + u_D)^2}{u_{\textrm{a}}}\omega^{N,M}_i\omega^{N,M}_j
 + \dfrac{(v^N_{0} + v_D)}{2}\Big(\omega^{N,M}_i\chi^{N,M}_j + \omega^{N,M}_j\chi^{N,M}_i\Big)\Big]d\sigma,\\\notag
&\quad\mbox{for $i,j=1,\ldots,d$}
\end{align}
for all ($\varphi,\vartheta,\phi,\psi$) of the form
\begin{align*}
&\varphi(x) = \sum_{j=1}^{N}a_j\xi^{\rm u}_j(x),~\vartheta(x) =  \sum_{j=1}^{N}a_j\xi^{\rm v}_j(x),\\
&\phi(x,y)=\sum_{j=1}^N\sum_{k=1}^{M}b_{jk}\zeta^{\chi}_{jk}(x,y),~\psi(x,y)=\sum_{j=1}^N\sum_{k=1}^{M}b_{jk}\zeta^{\omega}_{jk}(x,y),
\end{align*}
and 
\begin{align}\label{DiscInit}
&\U^N_j(0) := \int\limits_{\Om}\big(u_I-u_D(0)\big)\xi_j dx,\nonumber\\
&\V^N_j(0) := \int\limits_{\Om}\big(v_I-v_D(0)\big)\xi_j dx.
\end{align}
In \eqref{Disc1}-\eqref{Disc2}, we take as test functions $\varphi=\xi^{\rm u}_j,\vartheta=\xi^{\rm v}_j, \phi=\zeta^{\chi}_{jk}$ and $\psi=\zeta_{jk}^{\omega},$ for $j=1,\dots,N,k=1,\ldots,M$ and obtain the following system of ordinary differential equations for the coefficients $\U^N=\Big(\U^N_j\Big)_{j=1,\ldots,N}, \V^N=\Big(\V^N_j\Big)_{j=1,\ldots,N}$ and algebraic equations for the coefficients $\X=\Big(\X_{jk,p}\Big)_{j=1,\ldots,N, k=1,\ldots,M}$, and $\W=\Big(\W_{jk,p}\Big)_{j=1,\ldots,N,k=1,\ldots,M}, p=1,\ldots,d:$
\begin{align}
\label{DiscU}
&c\pt \U^N(t) + \sum^{N}_{i,j=1}\cA_{ij}\Big(\U^N(t),\V^N(t)\Big)\U^N_j(t) = F\Big(\U^N(t),\V^N(t)\Big),\\
\label{DiscV}
&\theta\pt \V^N(t) + \sum^{N}_{i,j=1}\cB_{ij}\Big(\U^N(t),\V^N(t)\Big)\V^N_j(t) = G\Big(\U^N(t),\V^N(t)\Big),\\
\label{DiscXW}
&\sum^{N}_{j=1}\sum^{M}_{k=1}\Big(\M^{G}_{jk} + \M^{S}_{jk}\Big)\X_{jk,p}+\sum^{N}_{j=1}\sum^{M}_{k=1}\N_{jk}\W_{jk,p}  +  \widetilde{R}\big(\X,\W\big) = -\big(F^{G} +  F^{S} + \widetilde{G}\big),
\end{align}
where for $i,j,l=1,\ldots,N, k=1,\ldots, M, p=1,\ldots, d,$ we have
\begin{align}
&\cA_{ij}=\int\limits_{\Om}\La\big(u^N_0 +u_D,v^N_0+v_D\big)\Na\xi^{\rm u}_{i}(x)\Na\xi^{\rm u}_j(x)~dx,\\
&\cB_{ij}=\int\limits_{\Om}\D\big(u^N_0 +u_D,v^N_0+v_D\big)\Na\xi^{\rm v}_{i}(x)\Na\xi^{\rm v}_j(x)~dx,\\
&F_j = \int\limits_{\Om}\bigg(\Na\!\cdot\!\Big(\La\big(u^N_{0} + u_D,v^N_{0} + v_D\big)\Na u_D(t)\Big) - c\pt u_D(t)\bigg)~\xi^{\rm u}_j(x)~dx,\\
&G_j =  \int\limits_{\Om}\bigg(\Na\!\cdot\!\Big(\D\big(u^N_{0} + u_D,v^N_{0} + v_D\big)\Na v_D(t)\Big) - \theta\pt v_D(t)\bigg)~\xi^{\rm v}_j(x)~dx,\\
&\Big(\M^{G}_{jk}\Big)_{il} = \int\limits_{\Om\times Y_{\rm g}}\La_{\rm g}(y)\Na_y\zeta^{\chi}_{jk}(x,y)\Na_y\zeta^{\chi}_{il}(x,y)~dydx, \notag \\
& \Big(\M^{S}_{jk}\Big)_{il} = \int\limits_{\Om\times Y_{\rm s}}\La_{\rm s}(y)\Na_y\zeta^{\chi}_{jk}(x,y)\Na_y\zeta^{\chi}_{il}(x,y)~dydx,\\
&\Big(F^{G}_p\Big)_{il}=\int\limits_{\Om\times Y_{\rm g}}\La_{\rm g}(y)~e_p\!\cdot\!\Na_y\zeta^{\omega}_{il}(x,y)dydx,\notag \\ 
&\Big(F^{S}_p\Big)_{il}=\int\limits_{\Om\times Y_{\rm g}}\La_{\rm g}(y)~e_p\!\cdot\!\Na_y\zeta^{\omega}_{il}(x,y)dydx\\
&\Big(\N_{jk}\Big)_{il} = \int\limits_{\Om\times Y_{\rm g}}QD(y)\Na_y\zeta^{\omega}_{jk}(x,y)\Na_y\zeta^{\omega}_{il}(x,y)~dydx,~\notag \\
&\Big(\widetilde{G_{p}}\Big)_{il}=\int\limits_{\Om\times Y_{\rm g}}QD(y)~e_p\!\cdot\!\Na_y\zeta^{\omega}_{il}(x,y)dydx\\
& \widetilde{R}_{jk,p}=Q\int\limits_{\Om}\int\limits_{\Gamma}f(u^N_{0} + u_D)\cH\Big(u^N_0+u_D,v^N_0+v_D,\chi^{N,M}_p,\omega^{N,M}_p\Big)\Big(\zeta_{jk}^{\chi}(x,y)-\zeta^{\omega}_{jk}(x,y)\Big)~d\sigma dx. 
\end{align}
The Cauchy problem \eqref{DiscInit}--\eqref{DiscV} admits a unique solution $(\U^N(t),\V^N(t))$ in $C^1([0,T])^N\times C^1([0,T])^N$ since by Lemma \ref{boundedTensor} and the regularities of the data, the functions $\cA, \cB, F$ and $G$ are globally Lipschitz continuous. The uniqueness of \eqref{DiscXW} follows from standard arguments for showing the wellposedness of discrete elliptic problems using properties ($H2$)-($H4$).
\subsection{Uniform estimates for the discretized problems}\label{subsect:uniformestimatesforthediscretizedproblems}
Here, we prove uniform estimates for the solutions of the finite dimensional problems. These estimates equip us with the necessary tool to pass in \eqref{Disc1}--\eqref{Disc2-D} to the limit $N,M \rightarrow \infty.$
\begin{theorem}\label{thm:almostcompactnesstheorem}
Let $P^N_x,P^N_y$ defined in \eqref{proj1} and \eqref{proj3} be stable with respect to the $L^2$-norm and $H^2$-norm and satisfies the assumptions $(H1)-(H6).$ Then there exists a constant $C>0$ independent of $N$ such that 
  \begin{align}
    & \|u_0^N \|_{L^\infty((0,T),H^1(\Om))} + \| \partial_t u_0^N \|_{L^2((0,T),L^2(\Om))}  \leq C , \label{uniformestimate1}\\
    &  \|v_0^N \|_{L^\infty((0,T),H^1(\Om))} + \| \partial_t v_0^N \|_{L^\infty((0,T),L^2(\Om))} \leq C, \label{uniformestimate2} \\
        & \|\chi^{N,M} \|_{L^2(\Om;H^1(Y))} \leq C,  \quad \| \omega^{N,M}\|_{L^2(\Om;H^1(Y_g))} \leq C, \label{uniformestimate3}   
  \end{align}
\end{theorem}
\begin{proof} We follow similar line of argument as in \cite{MunteanRadu:2010}. We take as test functions $(\varphi,\vartheta, \phi, \psi) = (u_0^N,v_0^N, \chi_{\alpha,j}^{N,M}, \omega_j^{N,M})$ with $\alpha = \{g,s \}$ in \eqref{Disc1}-\eqref{Disc2-D}. By applying the boundedness property of the diffusion tensors, we obtain
 \begin{align*}
  & \frac{c}{2}\frac{d}{dt}\|u_0^N(t)\|_{L^2(\Om)}^2 + \frac{\theta}{2}\frac{d}{dt}\|v_0^N(t)\|_{L^2(\Om)}^2 + \La_1 \| \nabla u_0^N \|_{L^2(\Om)}^2 
  + D_1 \| \nabla v_0^N \|_{L^2(\Om)}^2 \\
  &\leq \int\limits_{\Om} \bigg(\nabla\cdot\bigg(\La(u_0^N + u_D,v_0^N + v_D) \nabla u_D(t)\bigg) - c \partial_t u_D \bigg) u_0^N\,dx, &\quad\quad (\rm I) \\
  &+ \int\limits_{\Om}\bigg(\nabla\cdot\bigg(\D(u_0^N + u_D,v_0^N + v_D(t)) \nabla v_D(t) \bigg) - c \partial_t v_D(t)\bigg) v_0^N \bigg) dx,&\quad\quad (\rm II)
  \end{align*}
\begin{align*}
 & \|\nabla_y \chi_{g,j}^{N,M}\|^2_{L^2(\Om\times Y_{\rm g})} + \|\nabla_y \chi_{s,j}^{N,M}\|^2_{L^2(\Om\times Y_{\rm s})} + \|\Na_y \omega_j^{N,M}\|^2_{L^2(\Om\times Y_{\rm g})}\\
 & \leq -\int\limits_{\Om \times Y_{\rm g}} \La_{\rm g}e_j\nabla_y \chi_{g,j}^{N,M}~d\sigma dx -\int\limits_{\Om \times Y_{\rm s}} \La_{\rm s}e_j\nabla_y \chi_{s,j}^{N,M}~d\sigma  -Q\!\int\limits_{\Om \times Y_{\rm g}}De_j\nabla_y \omega_{g,j}^{N,M}~d\sigma\quad (\rm III)\\
 &+ Q\!\int\limits_\Om \int\limits_\Gamma f(u_0^N+u_D)\mathcal{H}(u_0^N+u_D,v_0^N+v_D,\chi_j^{N,M},\omega_j^{N,M})(\omega_j^{N,M} - \chi_j^{N,M} )\,d\sigma\quad (\rm IV)
 \end{align*}
 The right side of the equations above are numbered \mbox{$(\rm I)-(\rm IV)$}. From $(\rm III)$, we see that the integral vanishes due to the periodicity in $y$ of the functions $\chi$ and $\omega$. From $(\rm IV)$, we simplify the product of the functions $f$ and $\cH$ as $f(u^N_0 + u_D)(v^N_0 + v_D)\chi + \omega (u^N_0 + u_D)$ such that by the Lipschitz continuity property ($(H3)$ and $(H4)$), the integral reduces to
 \begin{align}\label{Est1}
&\leq C \int\limits_{\Om}\int\limits_{\Gamma}|u^N_0 + u_D|(|v^N_0 + v_D||\chi^{N,M}_j| + |\omega^{N,M}|)|\chi^{N,M}_j-\omega^{N,M}_j|~d\sigma dx\nonumber\\
&\leq C \int\limits_{\Om}\int\limits_{\Gamma}\Big(|u^N_0|^2 + |u_D|^2 + |v^N_0|^2 + |v_D|^2 + |\chi^{N,M}_j|^2 + |\omega_j^{N,M}|^2\Big)~d\sigma dx\nonumber\\
&\leq C\Big(\|u^N_0\|^2_{L^2(\Om)} + \|u_D\|^2_{L^2(\Om)}+\|v^N_0\|^2_{L^2(\Om)} + \|v_D\|^2_{L^2(\Om)}\Big)\nonumber\\ 
&+\delta \|\Na\chi^{N,M}_j\|^2_{L^2(\Om\times Y)} + C(\delta) \|\chi^{N,M}_j\|^2_{L^2(\Om\times Y)} +\delta^\prime \|\Na\omega^{N,M}_j\|^2_{L^2(\Om\times Y_{\rm g})} + C(\delta^\prime) \|\omega^{N,M}_j\|^2_{L^2(\Om\times Y_{\rm g})}.
\end{align}
The last expression on the right hand side of \eqref{Est1} is a consequence of the interpolation trace inequality. ($\rm I$) can be estimated as follows:
\begin{align}\label{Est3}
&\int\limits_{\Om}\Na\cdot\Big(\La(u^N_0+u_D,v^N_0+v_D)\Na u_D(t)\Big)-c\pt u_{D}(t)\Big)u^N_0 dx =\nonumber\\
&\int\limits_{\Om}\Big(\La_{u^N_0}\Na u^N_0\Na u_D(t) + \La_{u_D}\Na u_D(t)\cdot\Na u_D(t) + \La_{v^N_0}\Na v_0^N\cdot\Na u_D(t) +\La_{v_D}\Na v_D\cdot\Na u_D(t)\Big)u^N_0 dx\nonumber\\
&-\int\limits_{\Om}c\pt u_D(t)u^N_0~dx,\quad\mbox{where $\La_{\gamma} =\dfrac{\p\La}{\p\gamma}$}.
\end{align}
Since the derivatives of $\La(s,r)$ are bounded by virtue of Lemma \ref{boundedTensor}, we get
\begin{align}\label{Est3a}
\leq C\Big(\|\Na u^N_0\|^2_{L^2(\Om)} + \|\Na u_D(t)\|^2_{L^(\Om)} + \|\Na v^N_0\|^2_{L^(\Om)} + \|\Na v_D\|^2_{L^(\Om)} \Big) + \tilde{C}\Big(\|\pt u_D\|^2_{L^(\Om)} + \|u^N_0\|^2_{L^(\Om)} \Big).
\end{align}
Repeating the steps leading to \eqref{Est3a}, we obtain the estimate for $(\rm II)$. We choose \mbox{$\delta=\La_1/2$} and \mbox{$\delta^\prime=D_1/2$} in \eqref{Est1}, substituting on the left hand side of $(\rm III)$, using the regularity properties of $u_D$ and $v_D$, integrating all with respect to time and applying the Gronwall's inequality, we obtain
\begin{align*}
 & \|u_0^N\|_{L^2(\Om)}^2 + \|v_0^N\|_{L^2(\Om)}^2 
   +  \| \nabla u_0^N \|_{L^2(\Om)}^2 +  \| \nabla v_0^N \|_{L^2(\Om)}^2 \\
   & + \|\nabla_y \chi_{g,j}^{N,M} \|^2_{L^2(\Om \times Y_g) }
   + \|\nabla_y \chi_{s,j}^{N,M} \|^2_{L^2(\Om \times Y_s) }
   + \|\nabla_y \omega_j^{N,M} \|^2_{L^2(\Om \times Y_g) }  
   \leq C,
\end{align*}
for all $t \in (0,T)$ and $N \in \mathbb{N}$ with constant $C$ depending on the Dirichlet boundary conditions, initial boundary data, $T$, $\omega_j^{N,M}$ and  $\chi_j^{N,M}.$
Next, we derive $L^\infty$-estimates with respect to gradients in time for the heat and mass equation by testing the variational formulation \eqref{Disc1}--\eqref{Disc2-D} with $(\partial_t u_0^N , \partial_t v_0^N)$:
\begin{align}\label{Est4}
&\dfrac{1}{2}\dfrac{d}{dt}\int\limits_{\Om}c|\pt u^N_0(t)|^2 + \dfrac{1}{2}\dfrac{d}{dt}\int\limits_{\Om}\theta|\pt v^N_0(t)|^2 + \int\limits_{\Om}\dfrac{\p}{\p t}\Big(\La(u^N_0+u_D,v^N_0+v_D)\Na u^N_0\Big)\Na \pt u^N_0 dx \nonumber\\
&+ \int\limits_{\Om}\dfrac{\p}{\p t} \Big(\D(u^N_0+u_D,v^N_0+v_D)\Na v^N_0\Big)\Na \pt v^N_0 dx\nonumber\\
&= \int\limits_{\Om}\Big(\Na\cdot\Big(\dfrac{\p}{\p t}\Big(\La(u^N_0+u_D,v^N_0+v_D)\Na u^N_0(t)\Big)\Big)-c\pt u^N_0(t)\Big)\pt u^N_0~dx\nonumber\\
&+ \int\limits_{\Om}\Big(\Na\cdot\Big(\dfrac{\p}{\p t}\Big(\D(u^N_0+u_D,v^N_0+v_D)\Na v^N_0(t)\Big)\Big)-c\pt v^N_0(t)\Big)\pt v^N_0~dx\nonumber\\
\end{align}
The last term on the left hand side of \eqref{Est4} can be estimated as follows
\begin{align}\label{Est5}
&\int\limits_{\Om}\dfrac{\p}{\p t}\Big(\La(u^N_0+u_D,v^N_0+v_D)\Na u^N_0\Big)\Na \pt u^N_0 dx\nonumber\\ 
&= \int\limits_{\Om}(\La_{u^N_0}\pt u^N_0\Na u^N_0 + \La_{u_D}\pt u_D\Na u^N_0 + \La_{v^N_0}\pt v^N_0\Na u^N_0 + \La_{v_D}\pt v_D\Na u^N_0)\Na \pt u^N_0 dx\nonumber\\
&+ \int\limits_{\Om}\La(u^N_0+u_D,v^N_0+v_D)|\Na\pt u^N_0|^2 dx,\quad\mbox{where $\La_{\gamma} =\dfrac{\p\La}{\p\gamma}.$}
\end{align}
Applying the boundedness property of the derivatives of $\La(u,v)$ (Lemma \ref{boundedTensor}) and the regularities of the boundary data ensures that the first term on the right hand side of \eqref{Est5} is bounded. A similar kind of estimate is also obtained in terms of the variable $v^N_0.$ From the right hand side of \eqref{Est4}, we first integrate by parts in the higher order term and then differentiate the resulting expression with respect to time
\begin{align}\label{Est6}
 &\int\limits_{\Om}\Big(\Na\cdot\Big(\dfrac{\p}{\p t}\Big(\La(u^N_0+u_D,v^N_0+v_D)\Na u^N_0(t)\Big)\Big)-c\pt u^N_0(t)\Big)\pt u^N_0~dx\nonumber\\
 &=- \int\limits_{\Om}\Big(\La_{u^N_0}\pt u^N_0 + \La_{u_D}\pt u_D + \La_{v^N_0}\pt v^N_0 + \La_{v_D}\pt v_D\Big)\Na u_D(t)\Na\pt u^N_0 dx\nonumber\\
 &-\int\limits_{\Om}\Big(\La(u^N_0+u_D,v^N_0+v_D)\Na\pt u_D(t)\Na\pt u^N_0 + c\pt u_D(t)\pt u^N_0\Big)~dx.
\end{align}
Again, the first term in \eqref{Est6} is bounded by virtue of Lemma \ref{boundedTensor} and the regularities of $u_D$ and $v_D$. Integrating \eqref{Est4} with respect to time, using Lemma \ref{coercive} and the regularity properties of $u_D$ and $v_D$, we deduce
 \begin{align}\label{Est7}
    &  \| \partial_t u_0^N (t)\|_{L^2(\Om)}^2 + \| \partial_t v_0^N (t)\|_{L^2(\Om)}^2 
   + \| \nabla \partial_t u_0^N  (t) \|_{L^2(\Om)}^2 + \| \nabla \partial_t v_0^N (t) \|_{L^2(\Om)}^2 \nonumber\\
   & \leq  \| \partial_t u_0^N (0)\|_{L^2(\Om)}^2 + \| \partial_t v_0^N  (0)\|_{L^2(\Om)}^2 + C\Big(1 + \int\limits_0^t\!\!\int\limits_{\Om}|\pt u^N_0|^2 + \int\limits_0^t\!\!\int\limits_{\Om}|\pt v^N_0|^2\Big),
 \end{align}
 for $t \in (0,T)$ and $N \in \mathbb{N}.$ In \eqref{Est7}, the norms of the time derivatives at \mbox{$t=0$} need to be estimated. To this end, we evaluate the weak formulation \eqref{Disc1}--\eqref{Disc2-D} at \mbox{$t=0$} and test with $(\varphi,\vartheta)=(\partial_t u_0^N(0),  \partial_t v_0^N(0))$. This leads to
\begin{align}
&\int\limits_{\Om}c|\pt u^N_0(0)|^2 dx + \int\limits_{\Om}\La(u^N_0(0)+u_D(0),v^N_0(0)+v_D(0))\Na u^N_0(0)\Na\pt u^N_0(0) dx\nonumber\\
&+ \int\limits_{\Om}c|\pt v^N_0(0)|^2 dx + \int\limits_{\Om}\D(u^N_0(0)+u_D(0),v^N_0(0)+v_D(0))\Na v^N_0(0)\Na\pt v^N_0(0) dx\nonumber\\
&=-\int\limits_{\Om}\La(u^N_0(0)+u_D(0),v^N_0(0)+v_D(0))\Na u_D(0)\Na\pt u^N_0(0)dx -\int\limits_{\Om}c\pt u_D(0)\pt u^N_0(0) dx\label{rhs1}\\
&-\int\limits_{\Om}\D(u^N_0(0)+u_D(0),v^N_0(0)+v_D(0))\Na v_D(0)\Na\pt v^N_0(0)dx -\int\limits_{\Om}\theta\pt v_D(0)\pt v^N_0(0) dx, \label{rhs2}
\end{align}
where we have used integration by parts in \eqref{rhs1}-\eqref{rhs2}.  By using the boundedness property of the diffusion tensors, the regularity properties of the initial data and the Dirichlet boundary data, together with the stability of the projection operators $P^N_x$ and $P^N_y$ with respect to the $H^2-$norms, we obtain the appropriate bounds on the derivative at $t = 0$:
 \begin{align}
 \label{initialboundarydatabound2}
  \| \partial_t u_0^N(0)\|_{L^2(\Om)}^2 + \| \partial_t v_0^N(0)\|_{L^2(\Om)}^2  \leq C
 \end{align}
 By substituting \eqref{initialboundarydatabound2} into \eqref{Est7} and using Gronwall's inequality, we complete the proof of the theorem.
\end{proof}
The estimates in Theorem \ref{thm:almostcompactnesstheorem} provide the compactness of the solutions $(u_0^N, v_0^N, \chi^{N,M}, \omega^{N,M})$, which we require to pass to the limit as $N,M \rightarrow \infty$ in the nonlinear terms of the weak formulation \eqref{Disc1}--\eqref{Disc2-D}. 
\section{Convergence of the Galerkin approximation}\label{convergence}
In this section, we focus on the convergence of the Galerkin approximating vector function \mbox{$(u_0^N, v_0^N , \chi^{N,M}, \omega^{N,M})$} to the weak solution of \eqref{heatmassdiffusion} together with the cell problem \eqref{cellproblem}. Based on the uniform estimates established in Subsection~\ref{subsect:uniformestimatesforthediscretizedproblems}, we derive the convergence properties of the sequence of finite-dimensional approximations. The convergence theorem is stated in the sequel.
\begin{theorem}
\label{thm:convergenceresults1}
There exists a subsequence, still denoted by $(u^N_0, v^N_0,\chi^{N,M},\omega^{N,M})$ and a limit vector function $(u_0,v_0,\chi,\omega)\in L^2(0,T;H^1(\Om))\times L^2(0,T;H^1(\Om))\times L^2(\Om;[H^1_{\#}(Y)]^d)\times L^2(\Om;[H^1_{\#}(Y_{\rm g})]^d)$ with $(\pt u_0^N, \pt v_0^N) \in L^2((0,T)\times\Om))\times L^2((0,T)\times \Om)$ such that
\begin{align}
&(u^N_0, v^N_0,\chi^{N,M},\omega^{N,M})\rightarrow (u_0, v_0,\chi,\omega)  \label{eqn:convergenceresults1} \\
& \quad \text{weakly in } L^2(0,T;H^1(\Om))\times L^2(0,T;H^1(\Om)) \times L^2(\Om;[H^1_{\#}(Y)]^d) \times L^2(\Om;[H^1_{\#}(Y_{\rm g})]^d), \notag \\
&(\pt u^N_0, \pt v^N_0) \rightarrow (\pt u_0, \pt v_0)~\mbox{weakly in $L^2$},\label{eqn:convergenceresults2}\\
&(u^N_0, v^N_0,\chi^{N,M},\omega^{N,M})\rightarrow (u_0, v_0,\chi,\omega)~\mbox{strongly in $L^2,$} \label{eqn:convergenceresults3}\\
& \chi^{N,M}|_\Gamma \rightarrow \chi |_\Gamma,
\quad \text{and} \quad
\omega^{N,M}|_\Gamma \rightarrow \omega |_\Gamma
\quad \text{strongly in} \quad L^2((0,T) \times \Omega, L^2(\Gamma)).\label{eqn:convergenceresults2-1}
\end{align}
\end{theorem}
\begin{proof}
The estimates from Theorem~\ref{thm:almostcompactnesstheorem} immediately imply \eqref{eqn:convergenceresults1} and \eqref{eqn:convergenceresults2}. Since 
\begin{align*}
 \|u_0^N \|_{L^2((0,T);H^1(\Om))} + \| \partial_t u_0^N\|_{L^2( (0,T), L^2(\Om))}\leq C,
 \end{align*}
 and 
 \begin{align*}
 \|v_0^N \|_{L^2((0,T);H^1(\Om))} + \| \partial_t v_0^N\|_{L^2( (0,T), L^2(\Om))}\leq C.
\end{align*}
By the compactness theorem, see \cite[Chapter~4]{AdamsFournier:2003a}, there exist a subsequence such 
that 
\begin{align*}
  u_0^N \rightarrow u_0 \quad \text{and} \quad  v_0^N \rightarrow v_0, \quad \text{strongly in }
  \quad L^2((0,T) \times \Om).
\end{align*}
We note that Theorem~\ref{uniformestimate3} implies 
\begin{align*}
 \|\chi^{N,M} \|_{L^2(\Om;H^1_\#(Y))} + \|\omega^{N,M} \|_{L^2(\Om;H^1_\#(Y_g))} \leq C.
\end{align*}
Since the embedding 
\begin{align*}
 H^1(\Om, H^1(Y)) \hookrightarrow L^2(\Om,H^\beta(Y)), 
\end{align*}
is compact for every $1/2 < \beta < 1,$ it follows from Lion-Aubin's compactness theorem that there exist subsequences such that
\begin{align}\label{eqn:strongconvergence}
(\chi^{N,M}, \omega^{N,M}) \rightarrow (\chi, \omega)\quad\text{strongly in}\quad L^2((0,T) \times L^2(\Om,H^\beta(Y))),
\end{align}
for every $1/2 < \beta < 1.$ Due to the continuity of the trace operator 
\begin{align*}
 H^\beta(Y) \hookrightarrow L^2(\Gamma), \quad \text{for} \quad 1/2 < \beta < 1,
\end{align*}
yields \eqref{eqn:convergenceresults3} and \eqref{eqn:convergenceresults2-1}.
\end{proof}
}
\begin{theorem}
\label{thm:convergenceresults2}
Let the assumptions ($H5$)-($H6$) on the data be satisfied. Assume further that the projection operators $P_x^N$ and $P_y^M$ defined in \eqref{proj1} and \eqref{proj3} are stable with respect to the $L^2$-norm and $H^2$-norm. Let $(u_0,v_0,\chi,\omega)$ be the limit function obtained in Theorem~\ref{thm:convergenceresults1}. Then, the function $(u_0,v_0,\chi,\omega)=(u_0+u_D,v_0+v_D,\chi,\omega)$ is the unique weak solution of the problem \eqref{heatmassdiffusion}-\eqref{cellproblem}, and the whole sequence of Galerkin approximates converges.
\end{theorem}
\begin{proof}
By using Theorem~\ref{thm:convergenceresults1} and taking the limit of \eqref{Disc1}--\eqref{Disc2} for $N \rightarrow \infty$ and $M \rightarrow \infty,$ standard arguments yield the variational formulation \eqref{varfHomo}--\eqref{varfCell} for the function $(u_0,v_0,\chi,\omega)=(u_0+u_D,v_0+v_D,\chi,\omega).$
\end{proof}
\section{Simulation results}\label{numericalresults}
In this section, we present some numerical experiments to demonstrate the behavior of the quasilinear parabolic equations coupled to the cell problems, which possess a reaction-diffusion character. The numerical experiments are performed in the finite element software package, FreeFeM++ \cite{Hecht12}. The scheme involves a discretization in the space variable using triangular elements, followed by numerical integration in time using a semi-implicit Euler scheme. The discretized system is solved by using UMFPACK incorporated within the Freefem++ platform. From \eqref{heatmassdiffusion} and \eqref{cellproblem}, we know that the cell problems as well as the diffusion tensors depend on the values of the homogenized solutions $u$ and $v$. So, our first numerical example is to examine how $u$ and $v$ influence the behavior of the diffusion tensors. In essence, we seek for a phase $(u, v)$ diagram illustrating the behavior of the tensors for varying values of $(u, v)$. Figure \ref{Fig3} shows the structure of $\La_{11}$ and $\D_{11}$ as functions of $u$ and $v$. The evolution indicates that the effective diffusivity and the effective conductivity increase with temperature, but decrease with concentration. The results are calculated by solving \eqref{cellproblem} in the unit cell $Y=(0,1)\times (0,1)$ with a reference (circular) solid inclusion ($r=0.4$) for the heat equation or hole for the diffusion equation and using the parameter listing in Table \ref{Table1}. 
\begin{table}[!h]
\centering
\caption{Parameter values used in simulation.}
\begin{tabular}{lll}
\toprule
Parameter  	  &  Value   			  &     Description\\
\midrule
$\La_{\rm g}$      &  $2.38\times10^{-4}$     &	conductivity of gaseous part\\
$\La_{\rm s}$      &  $7\times10^{-4}$          &	conductivity of solid part\\
$D$		           &  $0.25$			  &    molecular diffusion\\
$c_{\rm g}$	   & $1.57\times 10^{-3}$     &    heat capacity of gas\\
$c_{\rm s}$	   & $0.69$   			  &   heat capacity of solid\\
$u_{\rm a}$          & $2.5$				  &	activation temperature\\
\bottomrule
\end{tabular}
\label{Table1}                               
\end{table}

\begin{figure}[!h]
\centering
 \begin{subfigure}[b]{0.48\textwidth}
\includegraphics[scale=0.15]{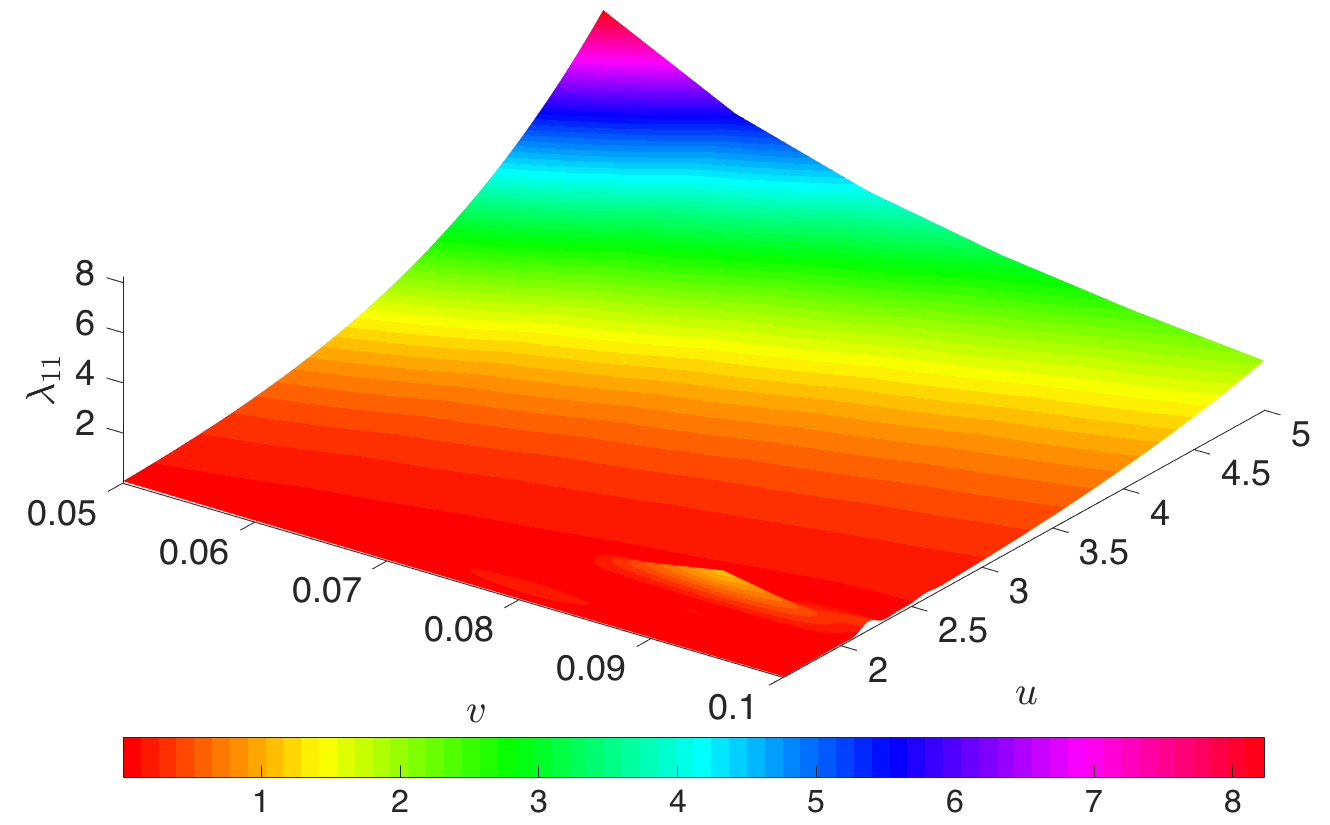}
\caption{}
\end{subfigure}
~
 \begin{subfigure}[b]{0.48\textwidth}
\includegraphics[scale=0.15]{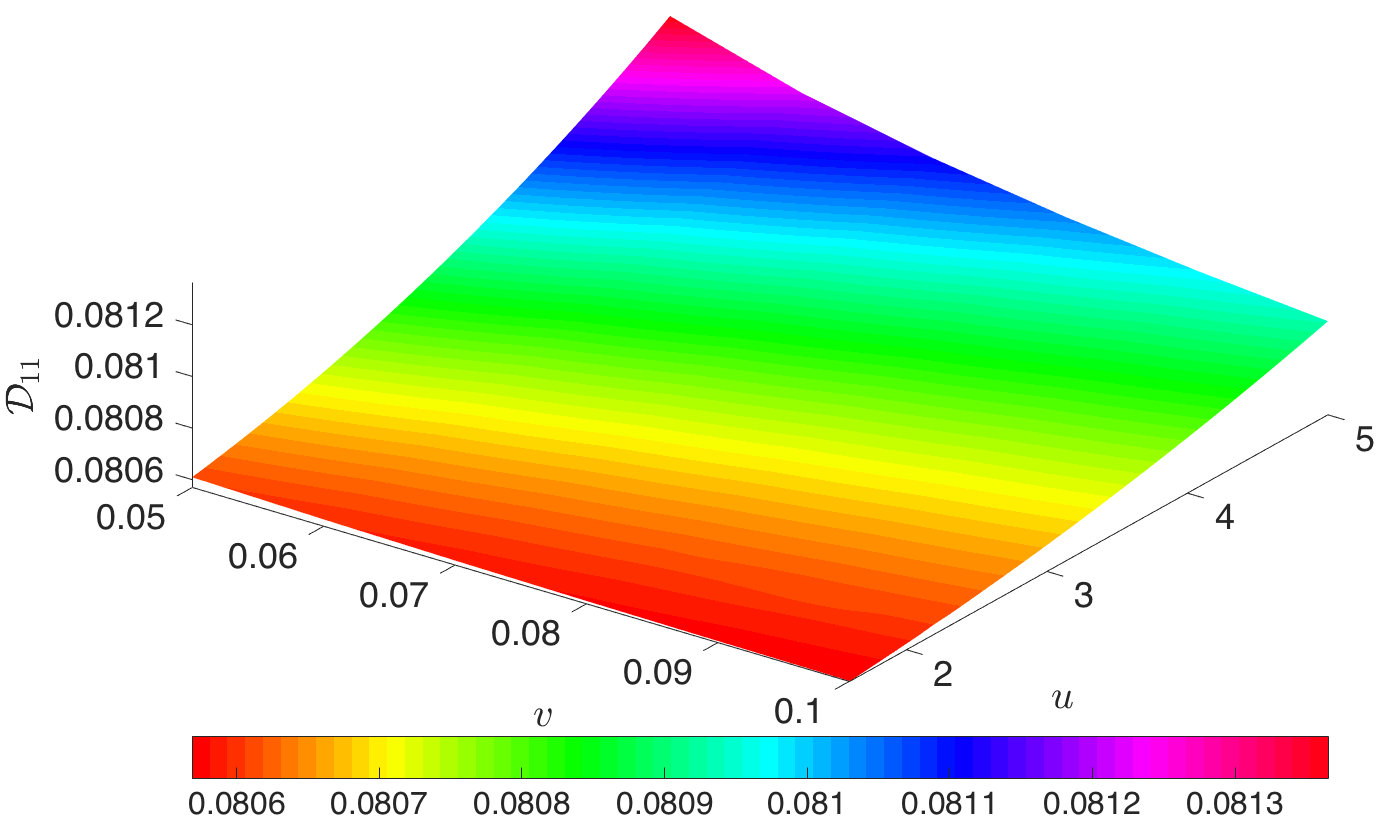}
\caption{}
\end{subfigure}
\caption{Structure of the effective diffusion tensors in a $(u, v)$-plane. (a) effective conductivity ($\La_{11}$) and (b) effective diffusivity ($\D_{11}$)  as functions of $u$ and $v$ for $A=5., Q=2.5$ and $u_{\rm a}=2.5$.}
\label{Fig3}
\end{figure}
Since the multiscale character of the studied macroscopic system is based on its weak coupling with the cell problem \eqref{cellproblem}, for each macroscopic point $x\in\Om$ and for each time, $t\in[0, T]$, \eqref{cellproblem} is solved in the microscale domain $Y$.~Consequently, a system of diffusion coefficients, \mbox{$\La(s, r):\Real\times\Real\rightarrow \Real^{d\times d}; \D(s, r):\Real\times\Real\rightarrow \Real^{d\times d}$} is simultaneously calculated in the process using formulas \eqref{Lmatrix} and \eqref{Dmatrix} respectively. Here, we take $d=2$. The reaction-diffusion behavior exhibited by \eqref{heatmassdiffusion} is a simple consequence of the reaction-diffusion processes taking place at the cell level. The response of the latter processes at the macroscale level is incorporated in the diffusion tensors $\La$ and $\D$, which are nonlocal and weakly anisotropic with respect to the variables, $u$ and $v$. \eqref{heatmassdiffusion} is solved in a rectangular domain $\Om=(0,5)\times (0,2.5)$.  In the simulation, all parameters are assumed to be nondimensional and real constants. The parameters for the microscale simulation are given in Table \ref{Table1}. The kinetic parameters ($Q, A$) and initial and boundary profiles of the field variables $u$ and $v$ are taken as free parameters, and hence reference to their particular values are given wherever necessary. It is worth mentioning that the effects of the kinetic parameters can also be assessed by examining their asymptotic behavior at the cell level.
\begin{figure}[!htp]
    \centering
    \begin{subfigure}[b]{0.45\textwidth}
        \includegraphics[width=\textwidth]{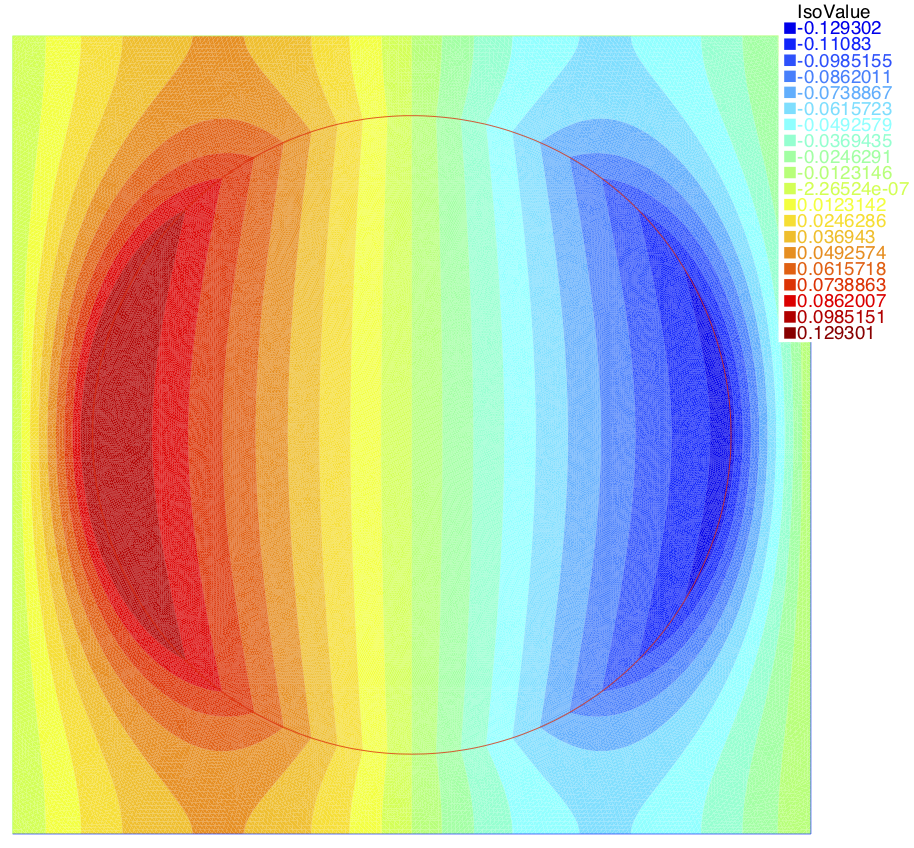}
        \caption{$\chi_1$ when $Q=0$ or $A=0$.}
        \label{figX1}
    \end{subfigure}
    ~ 
        \begin{subfigure}[b]{0.45\textwidth}
        \includegraphics[width=\textwidth]{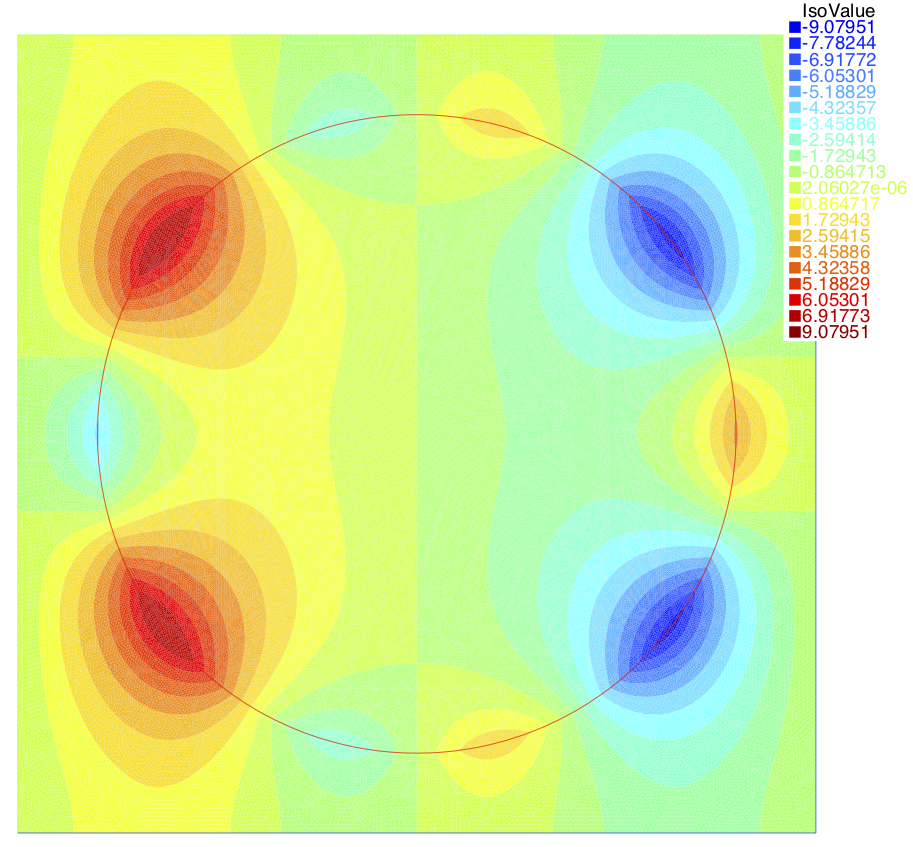}
        \caption{$\chi_1$ when $Q=0.1, u_D=5.0$ and $v_D=0.75$.}
        \label{figX2}
    \end{subfigure}
        ~ 
    \begin{subfigure}[b]{0.45\textwidth}
        \includegraphics[width=\textwidth]{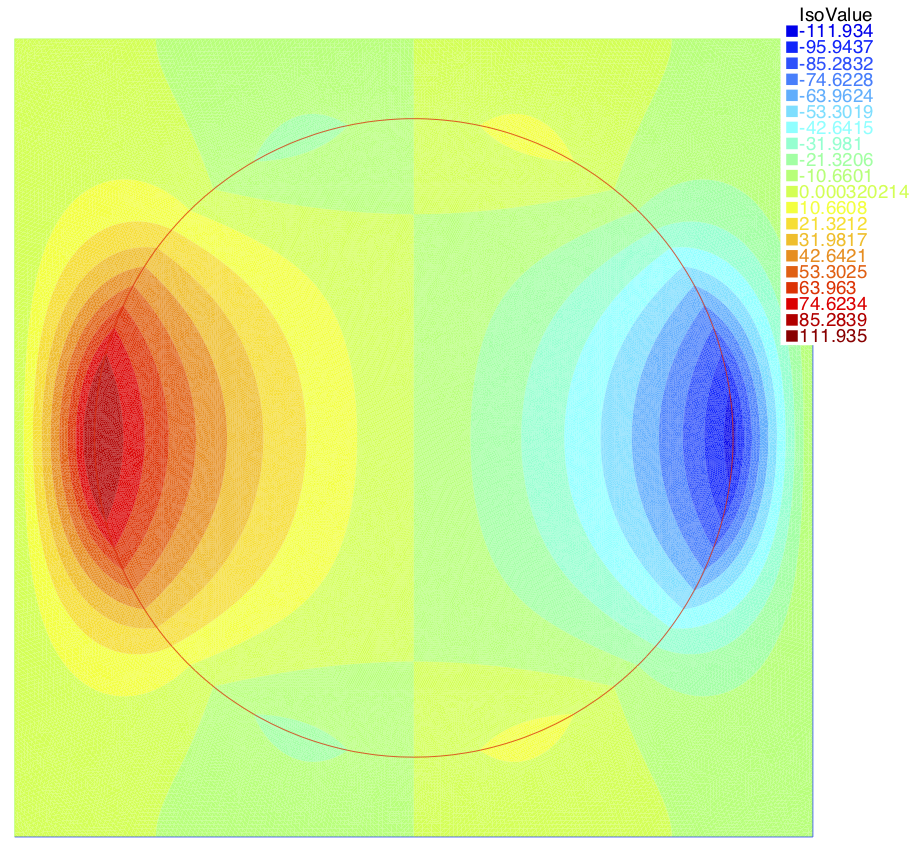}
        \caption{$\chi_1$ when $Q=1.0, u_D=5.0$ and $v_D=0.05$.}
        \label{figX3}
    \end{subfigure}
     ~ 
     \begin{subfigure}[b]{0.45\textwidth}
        \includegraphics[width=\textwidth]{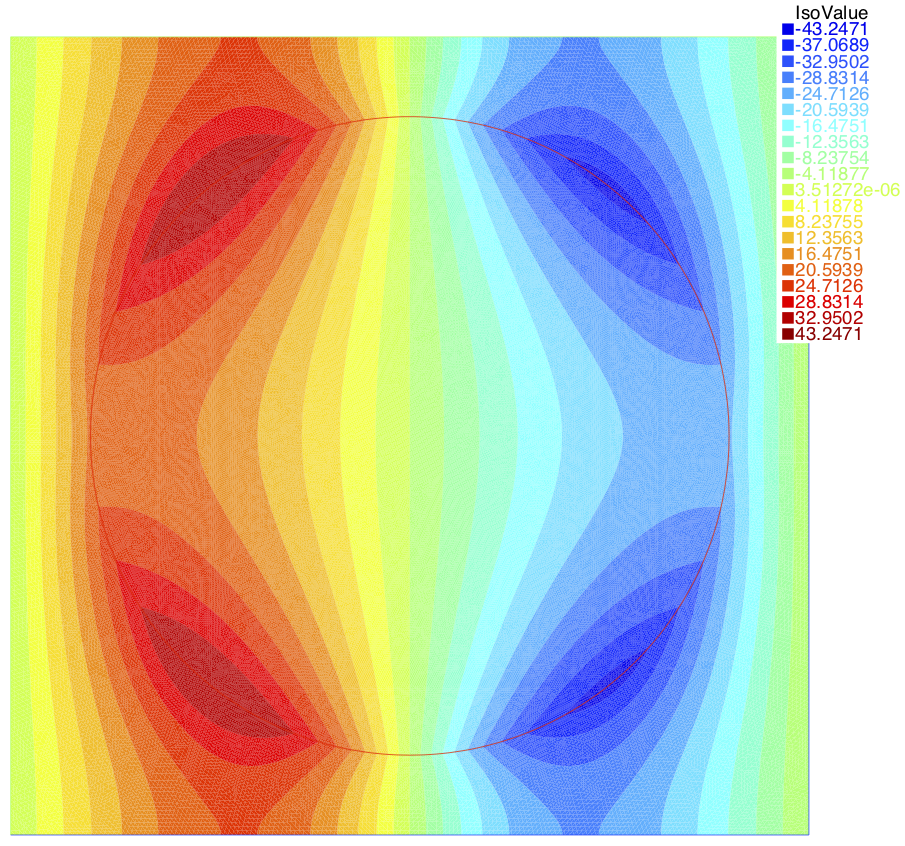}
        \caption{$\chi_1$ when $Q=2.5, u_D=5.0$ and $v_D=0.05$.}
        \label{figX4}
    \end{subfigure}
    \caption{Collage of cell solutions of $\chi_1$ corresponding to various combinations of $u_D, v_D$ and $Q$. Similar patterns of solutions are obtained for $\chi_2$.}
    \label{collage1}
\end{figure}

\begin{figure}[!htp]
    \centering
    \begin{subfigure}[b]{0.45\textwidth}
        \includegraphics[width=\textwidth]{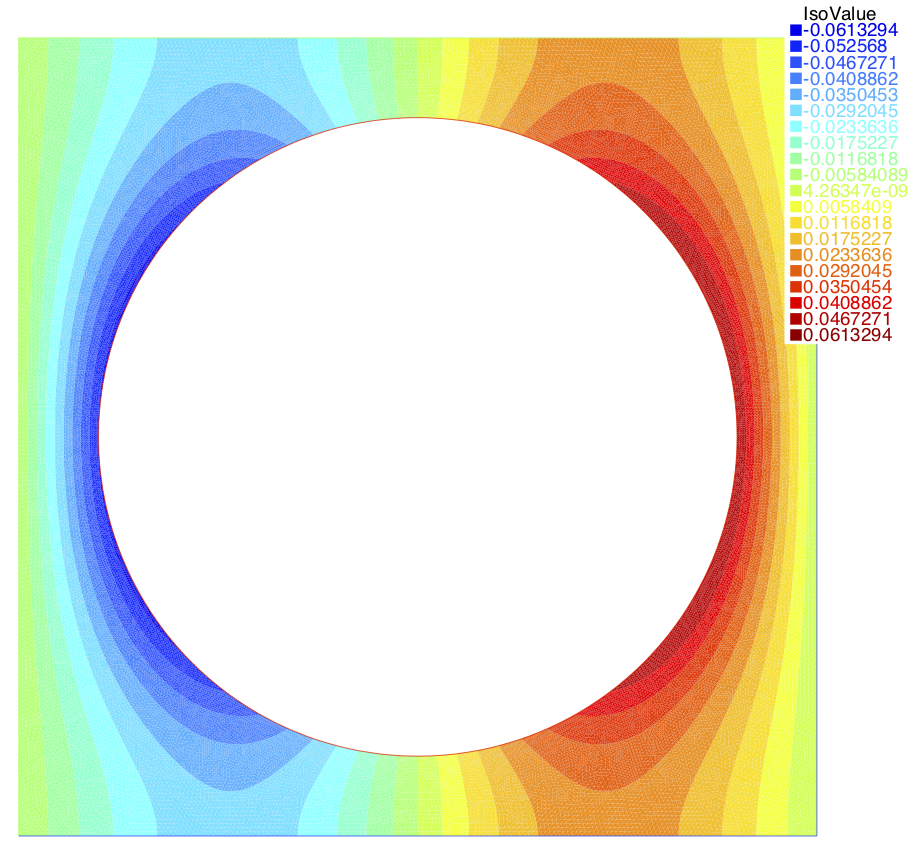}
        \caption{$\omega_1$ when $Q=0$ or $A=0$.}
        \label{figW1}
    \end{subfigure}
    ~ 
        \begin{subfigure}[b]{0.45\textwidth}
        \includegraphics[width=\textwidth]{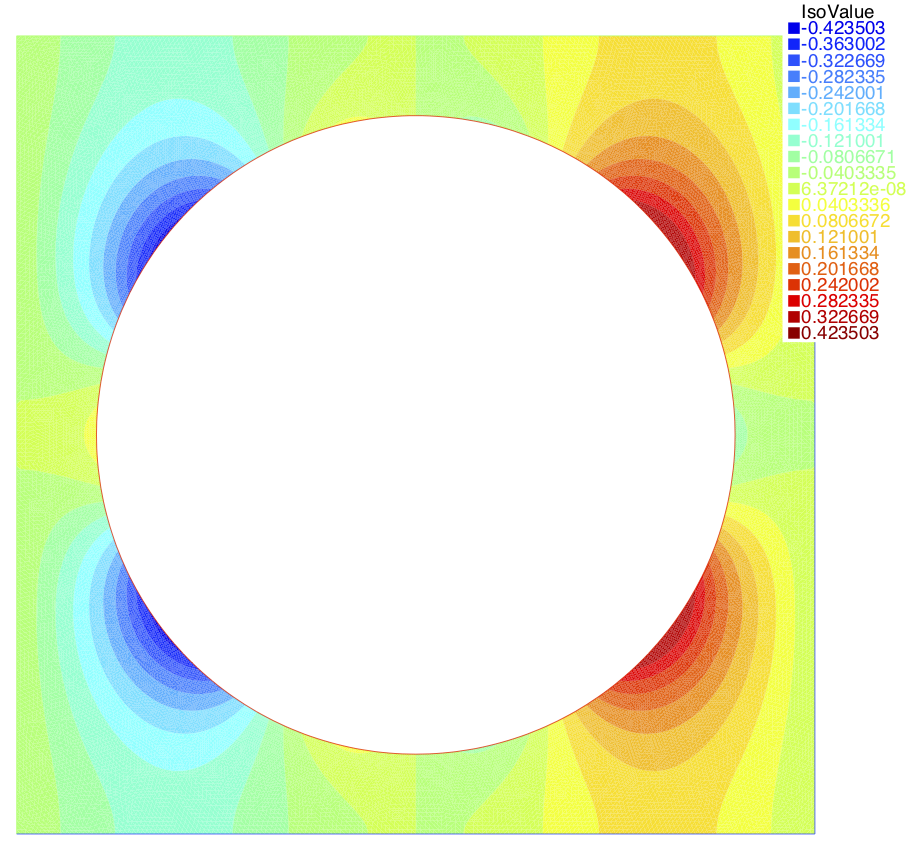}
        \caption{$\omega_1$ when $Q=0.1, u_D=5.0$ and $v_D=0.75$.}
        \label{figW2}
    \end{subfigure}
        ~ 
    \begin{subfigure}[b]{0.45\textwidth}
        \includegraphics[width=\textwidth]{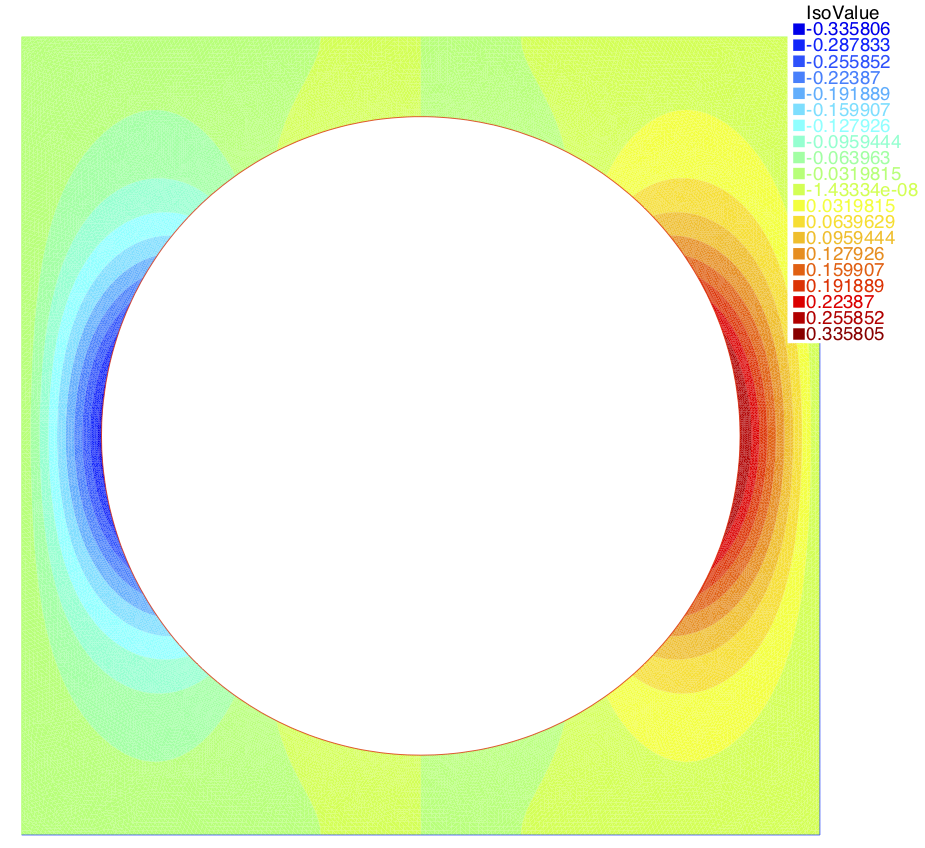}
        \caption{$\omega_1$ when $Q=1.0, u_D=5.0$ and $v_D=0.05$.}
        \label{figW3}
    \end{subfigure}
     ~ 
     \begin{subfigure}[b]{0.45\textwidth}
        \includegraphics[width=\textwidth]{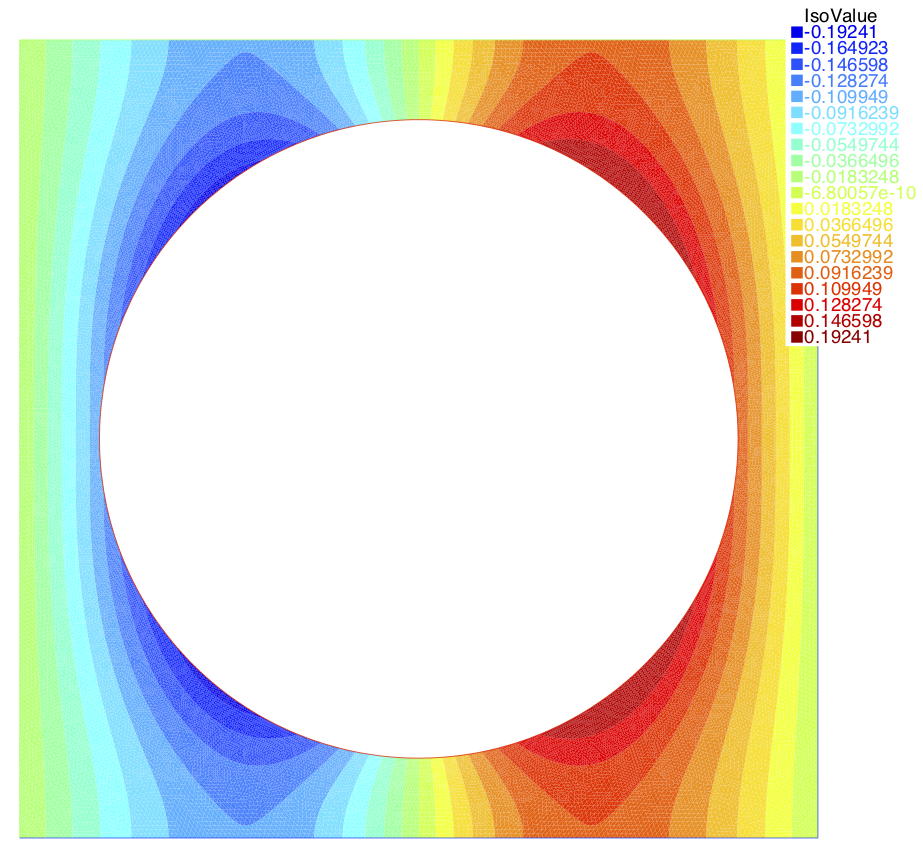}
        \caption{$\omega_1$ when $Q=2.5, u_D=5.0$ and $v_D=0.05$.}
        \label{figW4}
    \end{subfigure}
    \caption{Collage of cell solutions of $\omega_1$ corresponding to various combinations of $u_D, v_D$ and $Q$. Similar patterns of solutions are obtained for $\omega_2$. These solutions determine the boundary condition in the simulated upscaled model.}
    \label{collage2}
\end{figure}
We compute \eqref{heatmassdiffusion} based on various combinations in the data and parameters we used for the simulation of the cell solutions depicted in Figures \ref{collage1} and \ref{collage2}. The cell solutions show different patterns of solutions, which are determined by choices of the parameters $Q, u$ and $v$. When the pattern of solution is as depicted in Figures \ref{figX1} and \ref{figW1}, the cell problem is diffusion-dominated whereas in other cases, for instance, in Figures \ref{figX2}, \ref{figX3} and \ref{figX4}, the cell problem is characterized by a reaction-diffusion phenomenon. The boundary values $u_D$ and $v_D$ are chosen relative to the initial values in such a way that concentration gradients of $u$ and $v$ are initiated in the system. For the first set of simulations, the initial data is chosen as $u_I=1.7\mbox{ and } v_I=0.1$ while $u_D=5.0\mbox{ and } v_D=0.05$ are fixed to the left boundary. 
\begin{figure}[!htp]
\centering
\includegraphics[scale=0.65]{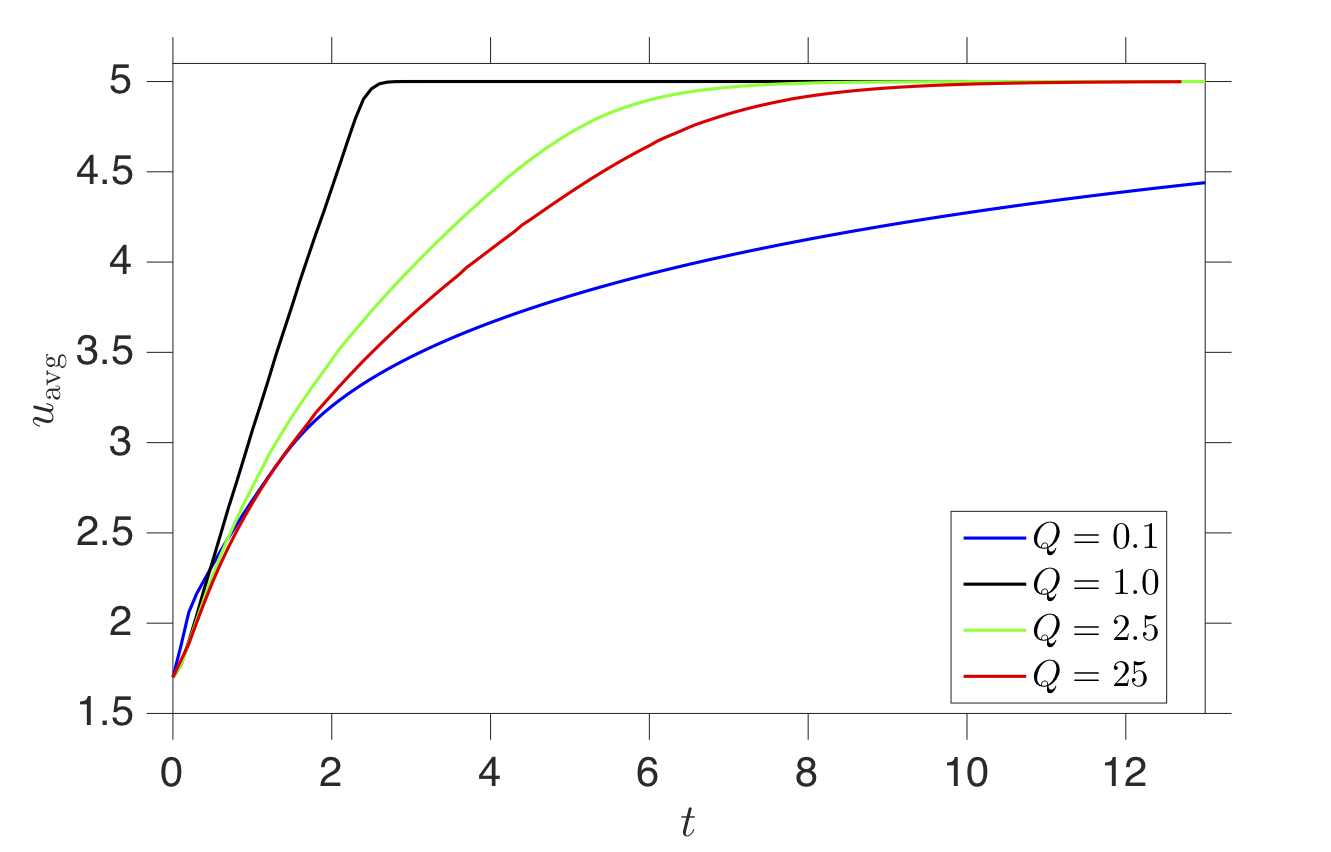}
\caption{Variations of the average temperature as a function of time and for various values of the heat release $Q$.}
\label{heatrelease}
\end{figure}
By varying the value of the heat release (\mbox{$Q=0.1, 1.0, 2.5$ and $25$}), we examine the behavior of the quasilinear system by calculating the average temperature defined as: 
\begin{align}
u_{\rm avg}(t)=\dfrac{1}{|\Om|}\int\limits_{\Om}u(t,x) dx.
\end{align}
\begin{figure}[!htp]
    \centering
    \begin{subfigure}[b]{0.4\textwidth}
        \includegraphics[scale=0.35]{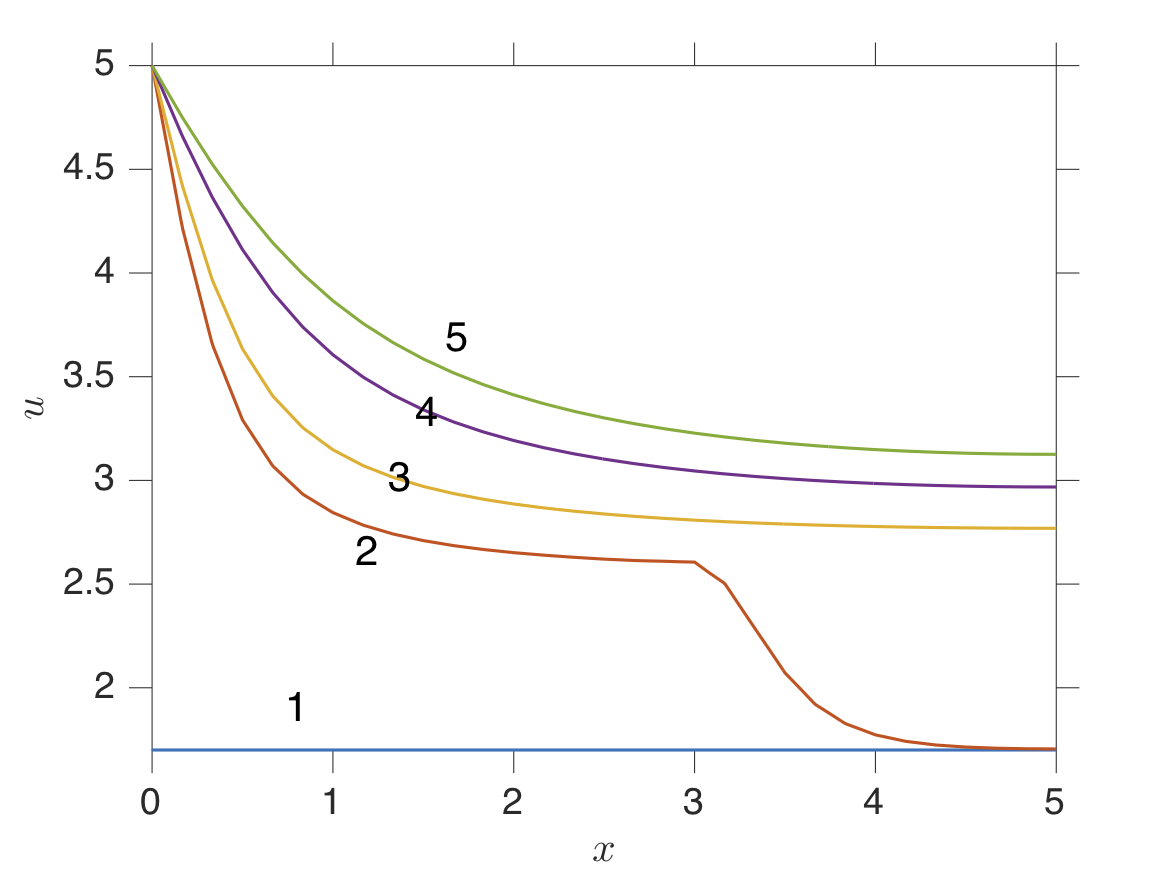}
        \caption{$Q=0.1.$}
        \label{}
    \end{subfigure}
    ~ 
        \begin{subfigure}[b]{0.4\textwidth}
        \includegraphics[scale=0.35]{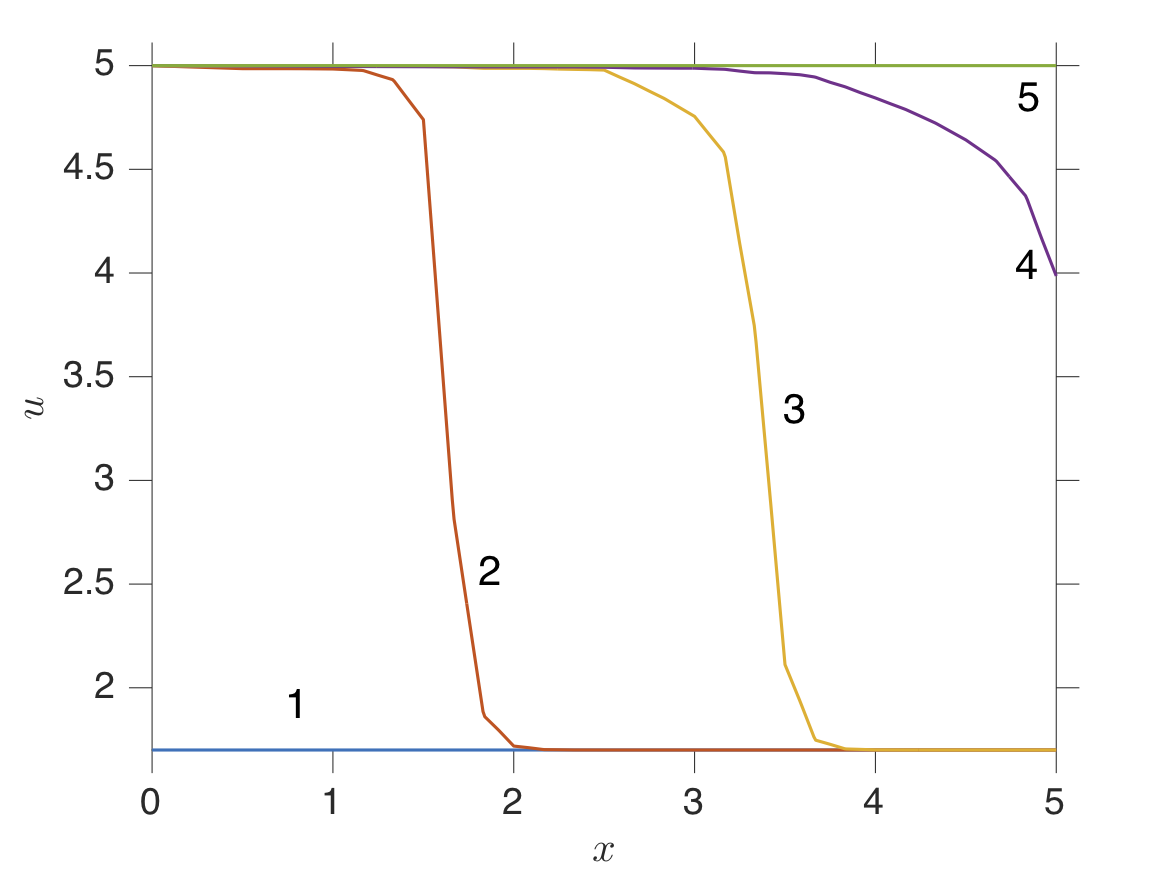}
        \caption{$Q=1.0.$}
        \label{fig:mouse}
    \end{subfigure}
        ~ 
    \begin{subfigure}[b]{0.4\textwidth}
        \includegraphics[scale=0.35]{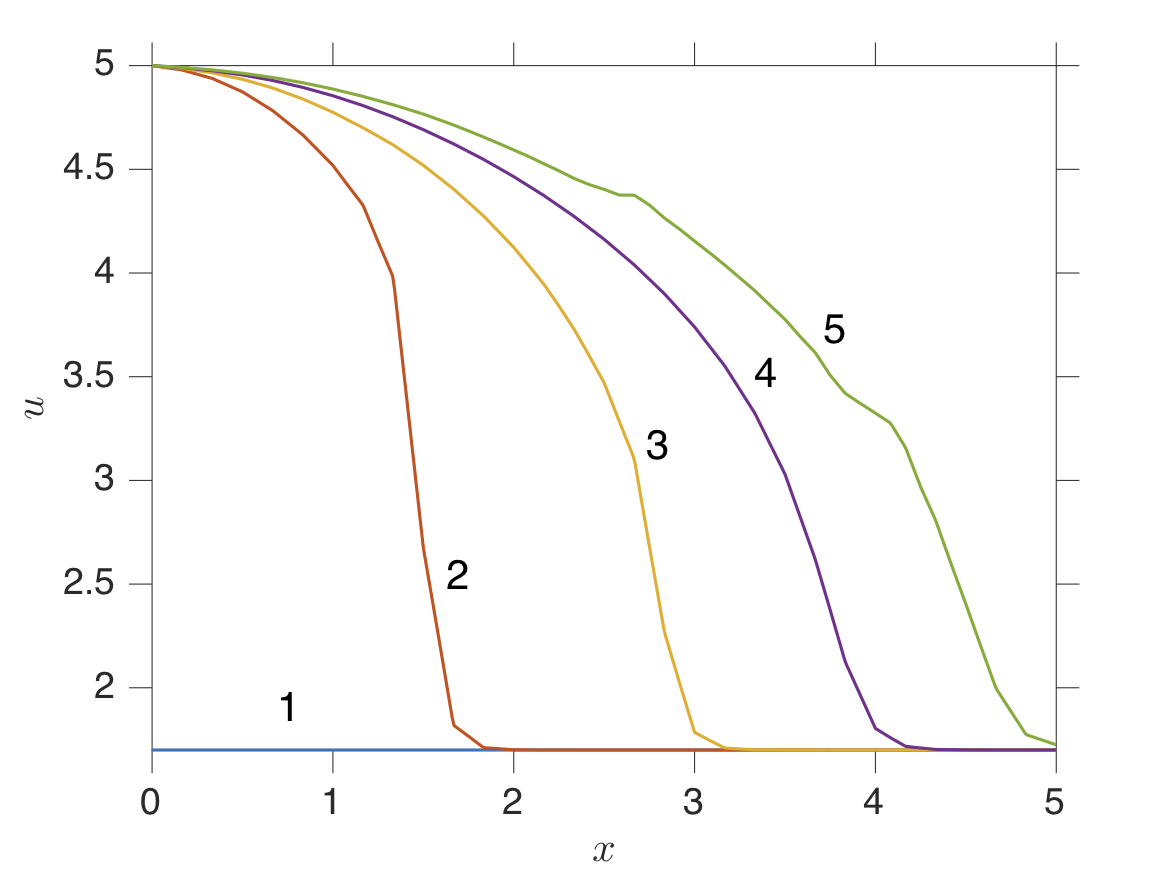}
        \caption{$Q=2.5.$}
        \label{fig:panther}
    \end{subfigure}
     ~ 
     \begin{subfigure}[b]{0.4\textwidth}
        \includegraphics[scale=0.35]{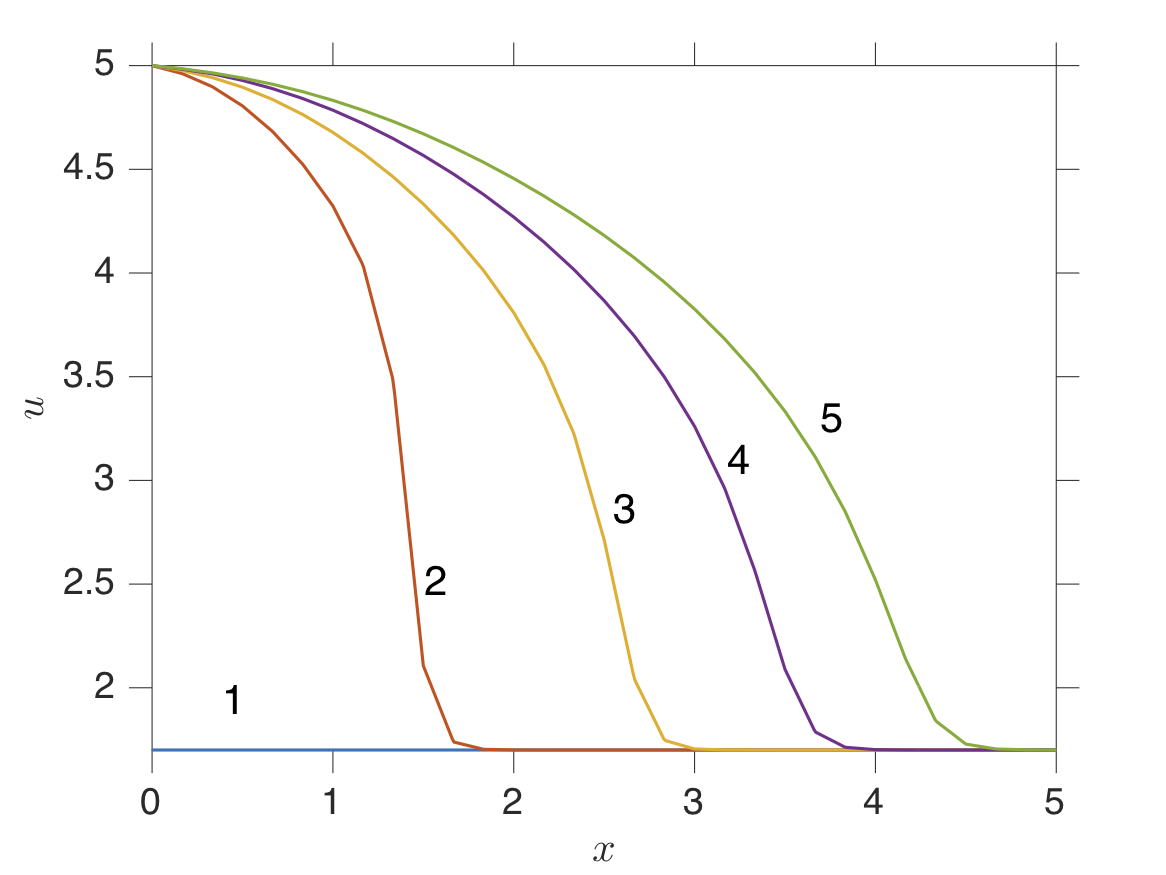}
        \caption{$Q=25.$}
        \label{fig:tiger}
    \end{subfigure}
    \caption{Evolution of the spatial profile of $u$ for $u_D=5., v_D=0.05, u_I=1.7, v_I=0.1$  and for varying values of $Q$. Distinct time points are labeled $1$-$5$ and correspond to the simulation times, $t=0, 0.8, 1.6, 2.4$ and $3.2.$}
    \label{average2}
\end{figure}
Figure \ref{heatrelease} depicts the average temperature as a function of time. It also indicates the behavior of the thermal front at distinct values of $Q$. Within the simulated values of $Q$, the time to reach steady state is shortest between $Q=0.1$ and $Q=1$ (Fig. \ref{average2}). Then, the time decreases for increasing $Q$. We point out that when either the heat release $Q$ or $A$ is zero (cf. Fig. \ref{figX1}), there is no front propagation. On the other hand, in the absence of an oxidizer, i.e. $v_I=v_D=0$, the system is purely diffusion dominated. Figures \ref{collage3} and \ref{collage4} depict the evolution of the spatial distributions of $u$ and $v$ at various simulation times. The front propagation (bottom to top) emanated from the reaction-diffusion processes at the cell level for each macroscopic point $x\in\Om$.
\begin{figure}[!ht]
    \centering
    \begin{subfigure}[b]{0.4\textwidth}
        \includegraphics[width=\textwidth]{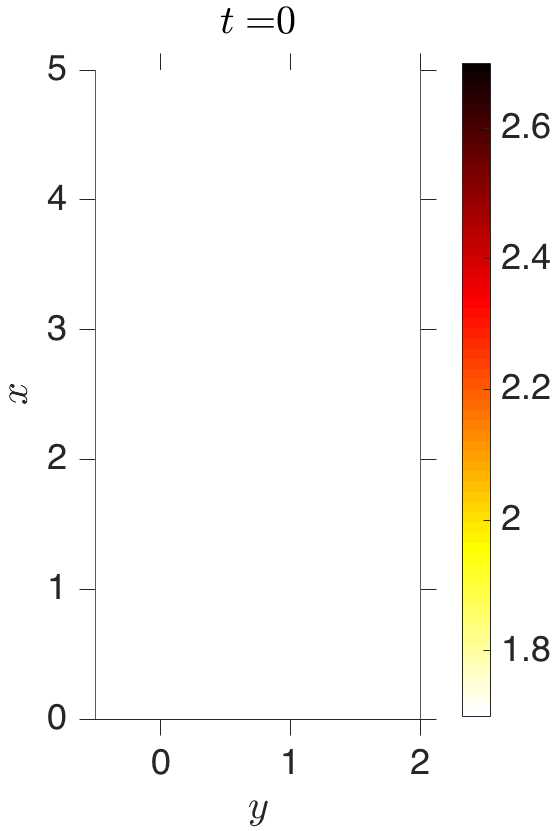}
        \caption*{}
        \label{}
    \end{subfigure}
    ~ 
        \begin{subfigure}[b]{0.4\textwidth}
        \includegraphics[width=\textwidth]{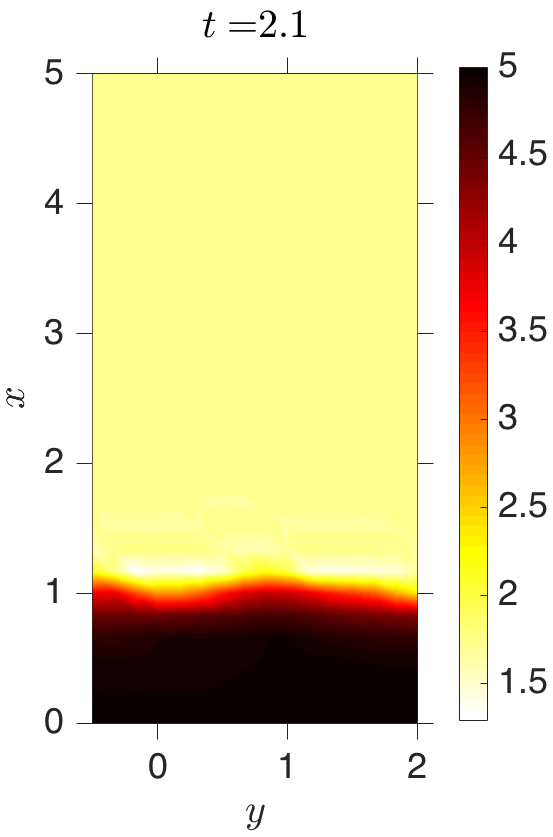}
        \caption*{}
        \label{}
    \end{subfigure}
        ~ 
    \begin{subfigure}[b]{0.4\textwidth}
        \includegraphics[width=\textwidth]{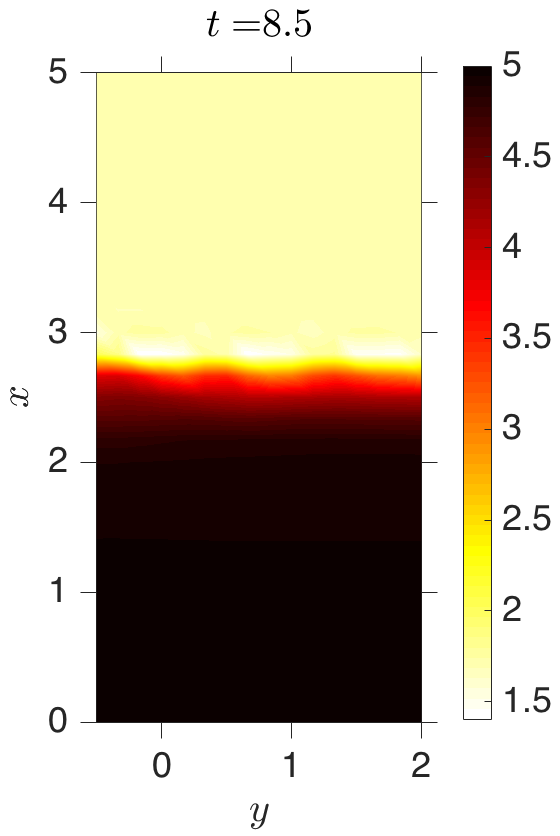}
        \caption*{}
        \label{}
    \end{subfigure}
     ~ 
     \begin{subfigure}[b]{0.4\textwidth}
        \includegraphics[width=\textwidth]{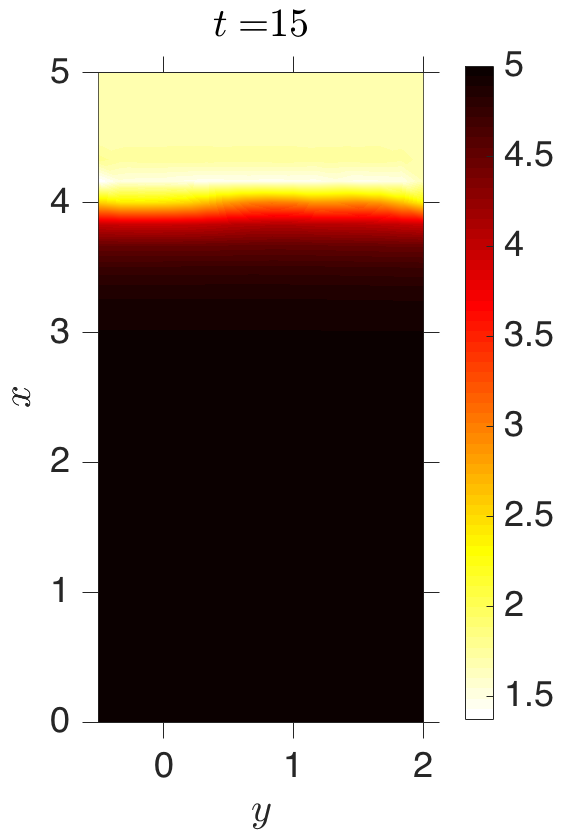}
        \caption*{}
        \label{}
    \end{subfigure}
    \caption{Evolution of the spatial distribution of temperature ($u$) for $u_D=5., v_D=0.75, u_I=1.7, v_I=1.0$ and $Q=0.1.$ Front propagation is from bottom to top.}
    \label{collage3}
\end{figure}

\begin{figure}[!ht]
    \centering
    \begin{subfigure}[b]{0.4\textwidth}
        \includegraphics[width=\textwidth]{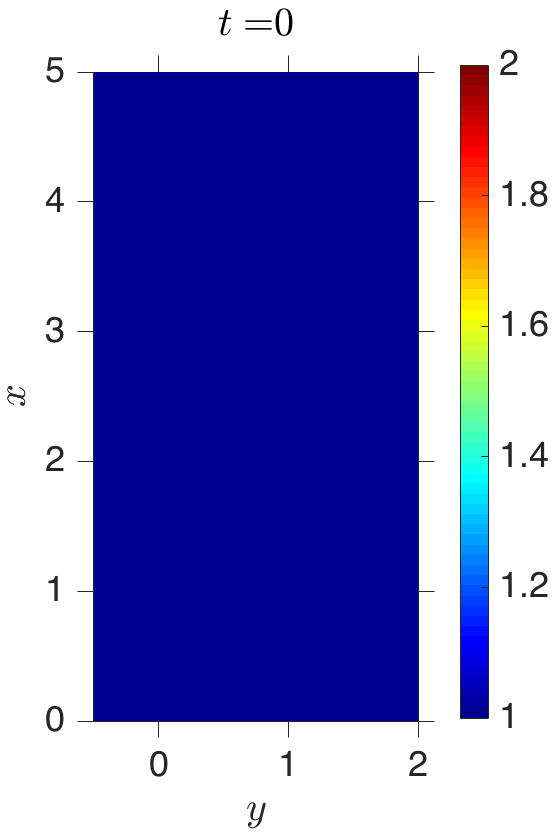}
        \caption*{}
        \label{}
    \end{subfigure}
    ~ 
        \begin{subfigure}[b]{0.4\textwidth}
        \includegraphics[width=\textwidth]{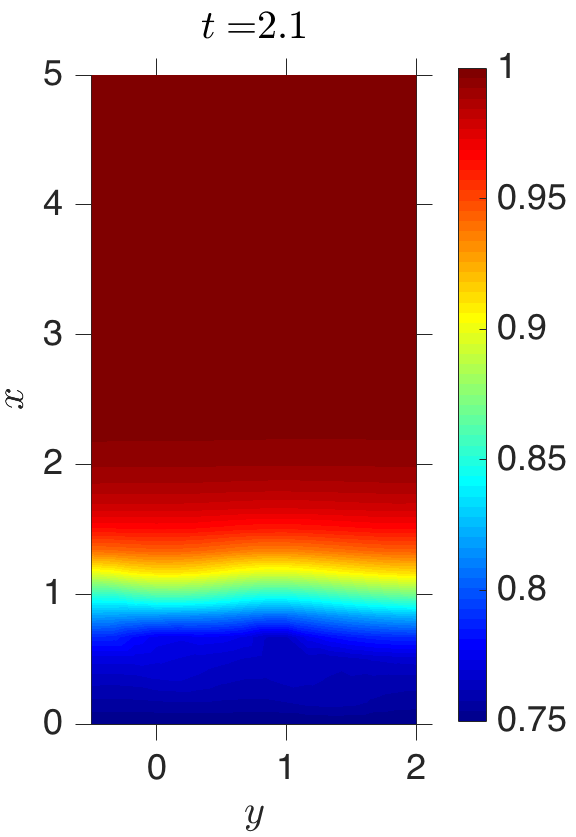}
        \caption*{}
        \label{}
    \end{subfigure}
        ~ 
    \begin{subfigure}[b]{0.4\textwidth}
        \includegraphics[width=\textwidth]{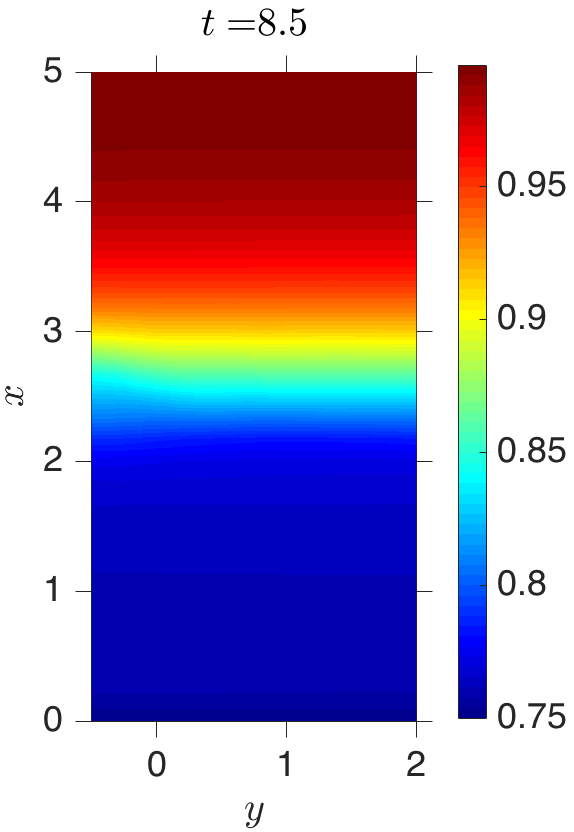}
        \caption*{}
        \label{}
    \end{subfigure}
     ~ 
     \begin{subfigure}[b]{0.4\textwidth}
        \includegraphics[width=\textwidth]{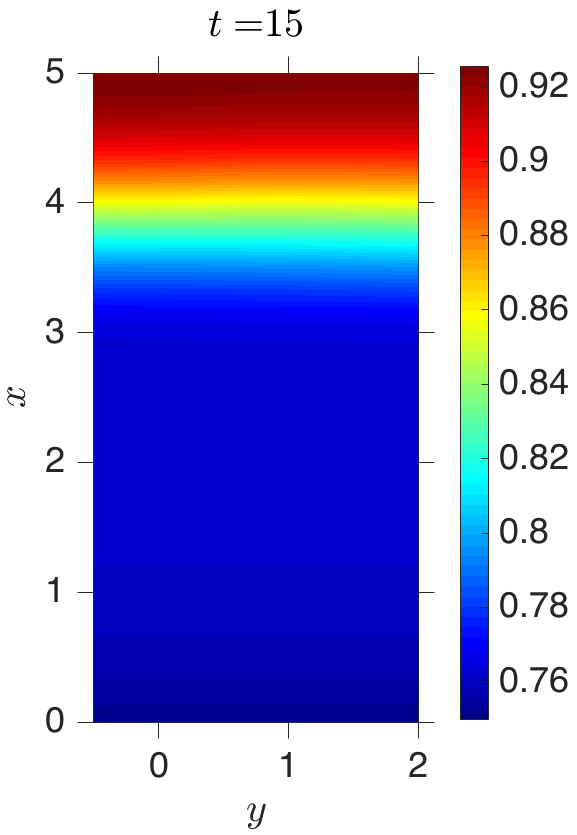}
        \caption*{}
        \label{}
    \end{subfigure}
    \caption{Evolution of the spatial distribution of concentration ($v$) for $u_D=5., v_D=0.75, u_I=1.7, v_I=1.0$ and $Q=0.1.$ Front propagation is from bottom to top.}
    \label{collage4}
\end{figure}
\section*{Acknowledgement}
This publication has emanated from research conducted with the financial support of Science Foundation Ireland (SFI) under Grant Number 14/SP/2750. The authors thank Adrian Muntean (Sweden) for useful comments and suggestions during the preparation of this paper. 


\begin{thebibliography}{10}

\bibitem{AdamsFournier:2003a}
R.~A. Adams and J.~J.~F. Fournier.
\newblock {\em Sobolev Spaces}.
\newblock Pure and Applied Mathematics. Elsevier Science, 2003.

\bibitem{Harsha}
G.~Allaire and H.~Hutridurga.
\newblock Upscaling nonlinear adsorption in periodic porous
  media--homogenization approach.
\newblock {\em Applicable Analysis}, 95(10):2126--2161, 2016.

\bibitem{Cioranescu99}
D.~Cior\u anescu and P.~Donato.
\newblock {\em An {I}ntroduction to {H}omogenization}.
\newblock Oxford University Press, New York, 1999.

\bibitem{Bakhvalov1989}
N.~Bakhvalov and G.~Panasenko.
\newblock {\em Homogenisation: {A}veraging {P}rocesses in {P}eriodic {M}edia},
  volume~36 of {\em Mathematics and its Applications (Soviet Series)}.
\newblock Kluwer Academic Publishers Group, Dordrecht, 1989.

\bibitem{Bensoussan78}
A.~Bensoussan, J.~L. Lions, and G.~Papanicolaou.
\newblock {\em Asymptotic {A}nalysis for {P}eriodic {S}tructures}, volume~5 of
  {\em Studies in Mathematics and its Application}.
\newblock North-Holland, 1978.

\bibitem{Andrez99}
A.~Burghardt, M.~Berezowski, and E.~W. Jacobsen.
\newblock Approximate characteristics of a moving temperature front in a
  fixed-bed catalytic reactor.
\newblock {\em Chem. Engineering and {P}rocessing}, 38:19 -- 34, 1999.

\bibitem{DODD12}
A.~B. Dodd, C.~Lautenberger, and C.~Fernandez-Pello.
\newblock Computational modeling of smolder combustion and spontaneous
  transition to flaming.
\newblock {\em Combustion and Flame}, 159(1):448 -- 461, 2012.

\bibitem{Blessing15}
B.~O. Emerenini, B.~A. Hense, C.~Kuttler, and H.~J. Eberl.
\newblock A mathematical model of quorum sensing induced biofilm detachment.
\newblock {\em PLOS ONE}, 10(7):1--25, 2015.

\bibitem{Hecht12}
F.~Hecht.
\newblock New development in freefem++.
\newblock {\em J. Numer. Math.}, 20(3-4):251--265, 2012.

\bibitem{Ijioma2015b}
E.~R. Ijioma, H.~Izuhara, M.~Mimura, and T.~Ogawa.
\newblock Computational study of non-adiabatic wave patterns in smoldering
  combustion under microgravity.
\newblock {\em East Asian Journal on Applied Mathematics}, 5(2):138 -- 149,
  2015.

\bibitem{Ijioma2015c}
E.~R. Ijioma, H.~Izuhara, M.~Mimura, and T.~Ogawa.
\newblock Homogenization and fingering instability of a microgravity smoldering
  combustion problem with radiative heat transfer.
\newblock {\em Combustion and Flame}, 162(10):4046 -- 4062, 2015.

\bibitem{Ijioma2016}
E.~R. {Ijioma} and A.~{Muntean}.
\newblock {Fast drift effects in the averaging of a filtration combustion
  system - a periodic homogenization approach}.
\newblock {\em ArXiv e-prints}, 2016.

\bibitem{Ijioma13}
E.~R. Ijioma, A.~Muntean, and T.~Ogawa.
\newblock Pattern formation in reverse smouldering combustion: {A}
  homogenisation approach.
\newblock {\em Combust. Theor. and Modell.}, 17(2):185--223, 2013.

\bibitem{KASHIWAGI92}
T.~Kashiwagi and H.~Nambu.
\newblock Global kinetic constants for thermal oxidative degradation of a
  cellulosic paper.
\newblock {\em Combustion and Flame}, 88(3):345 -- 368, 1992.

\bibitem{LadyzhenskayaSolonnikovUralceva:1968a}
O.~A. Ladyzhenskaya, V.~A. Solonnikov, and N.~N. Uraltseva.
\newblock {\em Linear and Quasilinear Equations of Parabolic Type}.
\newblock AMS, Providence, RI, 1968.

\bibitem{Lind18a}
M.~Lind and A.~Muntean.
\newblock A priori feedback estimates for multiscale reaction-diffusion
  systems.
\newblock {\em Numerical Functional Analysis and Optimization}, 39(4):413--437,
  2018.

\bibitem{Lind18b}
M.~Lind, A.~Muntean, and O.~M. Richardson.
\newblock Well-posedness and inverse {R}obin estimate for a multiscale
  elliptic/parabolic system.
\newblock {\em Applicable Analysis}, 97(1):89--106, 2018.

\bibitem{MAHMOUDI16}
A.~H. Mahmoudi, F.~Hoffmann, M.~Markovic, B.~Peters, and G.~Brem.
\newblock Numerical modeling of self-heating and self-ignition in a packed-bed
  of biomass using xdem.
\newblock {\em Combustion and Flame}, 163:358 -- 369, 2016.

\bibitem{MunteanRadu:2010}
A.~Muntean and M.~Neuss-Radu.
\newblock A multiscale {G}alerkin approach for a class of nonlinear coupled
  reaction-diffusion systems in complex media.
\newblock {\em Journal of Mathematical Analysis and Applications}, 371(2):705
  -- 718, 2010.

\bibitem{NeussRadu2007}
M.~Neuss-Radu and W.~J\"ager.
\newblock Effective transmission conditions for reaction-diffusion processes in
  domains separated by an interface.
\newblock {\em SIAM Journal on Mathematical Analysis}, 39(3):687--720, 2007.

\bibitem{RYU07}
C.~Ryu, A.~N. Phan, Y.~Yang, V.~N. Sharifi, and J.~Swithenbank.
\newblock Ignition and burning rates of segregated waste combustion in packed
  beds.
\newblock {\em Waste Management}, 27(6):802 -- 810, 2007.

\end{thebibliography}
\end{document}